\documentclass[11pt,a4paper]{amsart}

\usepackage{amsfonts,amsmath,amssymb,amsthm}

\usepackage[utf8]{inputenc}
\usepackage[T1]{fontenc}

\usepackage{graphicx}
\usepackage{float}
\usepackage{tikz}

\newcommand{\Cc}{\mathbb{C}} 
\newcommand{\Pp}{\mathbb{P}}
\newcommand{\Rr}{\mathbb{R}}
\newcommand{\Nn}{\mathbb{N}}
\newcommand{\Zz}{\mathbb{Z}}
\newcommand{\Qq}{\mathbb{Q}}
\newcommand{\Ff}{\mathbb{F}}
\newcommand{\Aa}{\mathbb{A}}

\renewcommand {\leq}{\leqslant}
\renewcommand {\geq}{\geqslant}

\newcommand{\point}{^\ast}
\newcommand {\xbar}{\underline{x}}

\newcommand {\lambdabar}{{\underline{\lambda}}}

\newcommand {\charac}{\mathop{\mathrm{char}}\nolimits}

{\theoremstyle{plain}
\newtheorem{theorem}{Theorem}[section]    
\newtheorem{lemma}[theorem]{Lemma}       
\newtheorem{proposition}[theorem]{Proposition}      
\newtheorem{corollary}[theorem]{Corollary}      
}
{\theoremstyle{remark}

\newtheorem{definition}[theorem]{Definition}      
\newtheorem*{remark*}{Remark}  
\newtheorem{remark}[theorem]{Remark}   

}

\def\degP{\rho}

\title[The Schinzel Hypothesis for Polynomials]{The Schinzel Hypothesis for Polynomials}

\author{Arnaud Bodin}
\author{Pierre D\`ebes}
\author{Salah Najib}

\email{arnaud.bodin@univ-lille.fr}
\email{pierre.debes@univ-lille.fr}
\email{slhnajib@gmail.com}

\address{Laboratoire Paul Painlev\'e, Math\'ematiques, Cit\'e Scientifique, Universit\'e de Lille, 59655 Villeneuve d'Ascq Cedex, France}
\address{Laboratoire Paul Painlev\'e, Math\'ematiques, Cit\'e Scientifique, Universit\'e de Lille, 59655 Villeneuve d'Ascq Cedex, France}
\address{Laboratoire ATRES, Facult\'e Polydisciplinaire de Khouribga, Universit\'e Sultan Moulay Slimane,
BP 145, Hay Ezzaytoune, 25000 Khouribga, Maroc.}



\thanks{{\it Acknowledgment}. 
This work was supported in part by the Labex CEMPI  (ANR-11-LABX-0007-01) 
and by the ANR project ``LISA'' (ANR-17-CE40–0023-01).
}

\date{\today}

\begin{document}

\begin{abstract}
The Schinzel hypothesis is a famous conjectural statement about primes
in value sets of polynomials, which generalizes the 
Dirichlet theorem about primes in an arithmetic progression.
We consider the situation that the ring of integers is replaced by a polynomial ring 
and prove the Schinzel hypothesis for a wide class 
of them: polynomials in at least one variable over the integers, polynomials in several variables over an arbitrary field,
 etc.
We achieve this goal by developing a version over rings of the Hilbert 
specialization property.
A polynomial 
Goldbach con\-jec\-ture 
is deduced, along with a result on spectra of rational functions.
\end{abstract}

\maketitle



\section{Introduction} \label{sec:intro}

The so-called Schinzel Hypothesis (H), which builds on an earlier conjecture of Bunyakovsky, 
was stated in  \cite{schinzel}. Consider a set $\underline P=\{P_1,\ldots,P_s\}$ of $s$ polynomials, 
irreducible in $\Zz[y]$, of degree $\geq 1$ and such that 
\vskip 1mm
 
 \noindent
 (*) {there is no prime $p\in \Zz$ dividing all values $\prod_{i=1}^s P_i(m)$, $m\in \Zz$.}
 \vskip 1mm
\noindent
Hypothesis (H) concludes that there are infinitely many $m\in \Zz$ 
such that $P_1(m),\ldots, P_s(m)$ are prime numbers. 
If true, the Schinzel hypothesis would solve many classical problems in number theory:
the twin prime problem (take ${\underline P} =\{y,y+2\}$), the infiniteness 
of primes of the form $y^2+1$ (take $\underline P=\{y^2+1\}$), the Sophie Germain prime problem ($\underline P=\{y, 2y+1\}$), etc.
However it is wide open except for one polynomial $P_1$ of degree one, in which case it
is the Dirichlet theorem about primes in an arithmetic progression.

We consider the situation that the ring $\Zz$ is replaced by a 
polynomial ring $R[\underline x]$ in $n\geq 1$ variables 
over some ring $R$, and 
``prime'' is understood as ``irreducible''. We prove the Schinzel Hypothesis in this situation for a 
wide class of rings $R$, for example $\Zz$, or $k[u]$ with $k$ an arbitrary field.
The  infiniteness of integers $m$ 
is replaced by a degree condition.

\subsection{Main result} \label{ssec:main-result}
Specifically, let $R$ be a Unique Factorization Domain {\rm (UFD)} with fraction field $K$. Our assumptions
include $K$ being a field with the product formula. Definition is recalled in \S \ref{sec:Hilbert}. 
The basic example is $K=\Qq$. The product formula is: $\prod_{p} |a|_p \cdot |a| = 1$
for every $a\in \Qq^\ast$, where $p$ ranges over all prime numbers, $|\cdot|_p$ is the $p$-adic absolute value 
and $|\cdot |$ is the standard absolute value. The rational function field $k(u_1,\ldots,u_r)$ in $r\geq 1$ variables over an arbitrary field $k$
is another example.

Given $n$ indeterminates $x_1,\ldots,x_n$, set $R[\underline x] = R[x_1,\ldots,x_n]$ ($n\geq 0$)\footnote{For $n=0$, we mean $R[\underline x]=R$, which is the original context of Schinzel's hypothesis.}.
Con\-sider $s\geq 1$ polynomials $P_1,\ldots,P_s$, irreducible in $R[\underline x,y]$, of degree $\geq 1$ in $y$.
Set $\underline P=\{P_1,\ldots,P_s\}$ and let ${\mathcal Irr}_{n}(R,\underline P)$ be the set of polynomials 
$M\in R[\underline x]$ such that $P_1(\underline x,M(\underline x)), \ldots, P_s(\underline x,M(\underline x))$  are irreducible in $R[\underline x]$.

For every $n$-tuple $\underline d=(d_1,\ldots,d_n)$ of integers $d_i\geq 0$, 
denote  the set of polynomials $M\in R[\underline x]$ such that $\deg_{x_i}(M) \leq d_i$, $i=1,\ldots,n$, by ${\mathcal Pol}_{R,n,\underline d}$. It is an affine space 
over $R$: the coordinates correspond  to the coefficients.

 \begin{theorem} \label{thm:schinzel1}
Assume that $n\geq 1$ and $R$ is a UFD with fraction field a field $K$ with the product formula, 
imperfect if $K$ is of characteristic $p>0$ {\rm (}i.e. $K^p\not=K${\rm )}.
For every $\underline d\in (\Nn^\ast)^n$ such that $\displaystyle d_1+\cdots+d_n \geq \max_{1\leq i\leq s} \deg_{\underline x}(P_i)+2$, 
\vskip -2pt

\noindent
the set ${\mathcal Irr}_{n}(R,\underline P)$ is Zariski-dense in ${\mathcal Pol}_{R,n,\underline d}$.
\end{theorem}

In particular, the following {\it Schinzel hypothesis for $R[\underline x]$} holds true: 
\vskip 1mm

\noindent
(**)  {\it there exist polynomials $M\in R[\underline x]$ with partial degrees any suitably large integers such that 
 $P_1(\underline x,M(\underline x)), \ldots, P_s(\underline x,M(\underline x))$ are irreducible in $R[\underline x]$
 \footnote{Up to adding $P_0=y$ to the set $\underline P$, one may also require that $M$ be irreducible in $R[\underline x]$.}. }
 \vskip 1mm
 
\noindent
 Irreducibility over $R$ is a main point. As a comparison, the Hilbert speciali\-zation property provides elements $m\in K$ such that 
 $P_1(\underline x,m), \ldots, P_s(\underline x,m)$ are 
irreducible over $K$
 (provided that 
 all $\deg_{\underline x}(P_i)$ are $\geq 1$).
Developing a {\it Hilbert property over rings} will in fact be the core of our approach; we say more about this in \S \ref{ssec:hilbertian_rings}.

Rings $R$ satisfying the assumptions of Theorem \ref{thm:schinzel1} include:
 
\noindent
 (a) the ring $\Zz$ of integers, and more generally, every ring ${\mathcal O}_k$ of integers of a number field $k$ of class number $1$, 
\vskip -0,3mm

\noindent
(b) every polynomial ring $k[u_1,\ldots,u_r]$ with $r\geq 1$ and $k$ an arbitrary field.

\noindent
(c) fields (so $R=K$) with the product formula, imperfect if of character\-istic $p>0$ e.g. $\Qq$, $k(u_1,\ldots,u_r)$ ($r\geq 1$, $k$ arbitrary), their finite extensions.

 As to the analog of assumption (*), it is automatically satisfied under our hypotheses
 (Lemma \ref{lemma:first}).
Our approach also allows the situation that the polynomials $P_i$ have several variables
$y_1,\ldots,y_m$, which leads to a multivariable Schinzel hypothesis for polynomials (Theorem \ref{thm:schinzel3}).

\subsection{Examples} Take $R[\underline x]$ as above and $P_i = b_i(\underline x) \hskip 0,5pt y^{\rho_i} + a_i(\underline x)$ with $\rho_i\in \Nn^\ast$, $a_i,b_i$ relatively prime in $R[\underline x]$ (possibly in $R$) and such that $-a_i/b_i$ satisfies the Capelli con\-dit\-ion that makes $b_iy^{\rho_i} + a_i$ irreducible in $K(\underline x)[y]$, \hbox{i.e.} $-a_i/b_i\notin  K(\underline x)^\ell$ for every prime divisor $\ell$ of $\rho_i$ and $-a_i/b_i\notin -4 K(\underline x)^4$
if $4\hskip 2pt |\hskip 2pt \rho_i$. Then
\vskip 1,5mm

\noindent
(***) {\it 
there exist polynomials $M\in R[\underline x]$ with partial degrees any suitably large integers such that
$b_1 M^{\rho_1}+a_1, \ldots, b_s M^{\rho_s}+a_s$ 
are irreducible in $R[\underline x]$.}

\vskip 1,5mm

\noindent
This solves the polynomial analogs of all famous number-theoretic problems mentioned above
(twin prime, etc.), and proves the Dirichlet theorem as well.

On the other hand, Schinzel's hypothesis for $R[\underline x]$ obviously fails (hence Theorem \ref{thm:schinzel1} too)  for $n=1$ 
if $R=K$ is algebraically closed. It also fails for the finite field $R=\Ff_2$ and $\underline P=\{y^8+x^3\}$: from an example of Swan \cite[pp.1102-1103]{swan},
$M(x)^8 + x^3$ is reducible in $\Ff_2[x]$ for every $M\in \Ff_2[x]$.
Interestingly enough, results of Kornblum-Landau \cite{kornblum} show that it does hold for $\Ff_q[x]$ in the degree one case and
for one polynomial, i.e., in the situation of the Dirichlet theorem; see also \cite[Theorem 4.7]{rosen}. The situation that $R=K$ is a finite field, and the related one that $R=K$ is a PAC field\footnote{A field $K$ is PAC if every curve over 
$K$ has infinitely many $K$-rational points. The first examples of PAC fields were ultraproducts of finite fields.}, and $n=1$, have led to valuable variants; see \cite{Bary-Soroker_Dirichlet},
\cite{liorSchinzel},  \cite{bender-wittenberg}.

\subsection{Special rings} 
The special situation that $R=K$ is a field is easier, and is dealt with in \S \ref{sec:preliminary}. In the addendum to Theorem  \ref{thm:schinzel1} (in \S \ref{sec:preliminary}), $K$ is assumed to be a Hilbertian field, more exactly a {\it totally Hilbertian} field 
(definitions are in \S \ref{ssec:reminder}). This provides more fields than those in \S \ref{ssec:main-result}(c) for which Theorem \ref{thm:schinzel1} holds (with $R=K$): every abelian extension of $\Qq$, the field $k((u_1,\ldots,u_r))$ of formal power series over a field $k$ in at least two variables, etc. 
\vskip 1mm

For $R=k[u]$ with $k$ a field, we have this version of Theorem \ref{thm:schinzel1} in which the partial degrees of $M$ are prescribed,
including the degree in $u$.

\begin{theorem} \label{thm:schinzel2}
With ${\underline P}$ as above and $n\geq 1$, assume $R=k[u]$ with $k$ an ar\-bi\-tra\-ry field. 
For every $\underline d\in (\Nn^\ast)^n$ satisfying $\displaystyle d_1+\cdots+d_n \geq \max_{1\leq i\leq s} \deg_{\underline x}(P_i)+2$, there is an integer $d_0\geq 1$ such that for every integer $\delta\geq d_0$, there is a polynomial $M\in {\mathcal Irr}_{n}(R,\underline P)$
satisfying
 
 \vskip 1mm

\centerline{
$\left\{\begin{matrix}
\deg_{x_j}(M)= \hskip 2mm d_j \hskip 6mm j=1,\ldots, n \hfill \\
\deg_u(M) = \left\{\begin{matrix}
\delta \hskip 3mm  &\hbox{if ${\rm char}(k) = 0$} \hfill \\
p\delta \hskip 3mm  &\hbox{if ${\rm char}(k) = p>0$}. \hfill \\
\end{matrix}
\right.\\
\end{matrix}\right.$
}
 \end{theorem}

Identifying $k[u][x_1,\ldots,x_n]$ with a polynomial ring in $n+1$ variables, it follows that Schinzel's hypothesis holds for polynomial rings in at least $2$ va\-ria\-bles over a field of characteristic $0$.
In characteristic $p>0$, a weak version holds where one degree is allowed to be any suitably large multiple of $p$. 
\vskip 1mm

In the degree one case of the Schinzel hypothesis, \hbox{i.e.} in the Dirichlet situation, one can get rid of this last restriction.

\begin{theorem} \label{thm:Main} Assume that $n\geq 2$ and $k$ is an arbitrary field.
Let $(A_1,B_1),$ $\ldots,$ $(A_s,B_s)$ be $s$ pairs of nonzero relatively prime polynomials in $k[\underline x]$. There is an integer $d_0\geq 1$ with this property: for all integers $d_1,\ldots,d_n$ larger than $d_0$, there exists an irreducible polynomial $M\in k[\xbar]$ such that $A_i+B_iM$ is irreducible in $k[\underline x]$, $i=1,\ldots,s$, 
and $\deg_{x_j}(M) = d_j$, 
$j=1,\ldots, n$.
\end{theorem}

To our knowledge, this was unknown, even for $s=1$. When $k$ is infinite, we have a stronger version,
not covered either by Theorems \ref{thm:schinzel1} 
 and \ref{thm:schinzel2}. Let $\overline k$ denote an algebraic closure of $k$.

\begin{theorem} \label{thm:Main2} Assume $n\geq 2$ and $k$ is an infinite field. Let $A,B\in k[\underline x]$ be two nonzero relatively prime polynomials and
 ${\mathcal Irr}_{n}(k,A,B)$ the set of polynomials $M\in k[\xbar]$ such that $A+BM$ is irreducible in $k[\xbar]$. For every $\underline d\in (\Nn^\ast)^n$, ${\mathcal Irr}_{n}(\overline k,A,B)$ 
contains a nonempty Zariski open subset of ${\mathcal Pol}_{k,n,\underline d}(k)$. 
\end{theorem}

\subsection{The Goldbach problem} \label{ssec:goldbach} The analog of the Goldbach conjecture for a polynomial ring $R[\underline x]$ is that every nonconstant polynomial ${\mathcal Q}\in R[\underline x]$ is the sum of two irreducible polynomials $F,G\in R[\underline x]$ with $\deg(F) \leq \deg({\mathcal Q})$ {\rm (}and so $\deg(G) \leq \deg({\mathcal Q})$ too{\rm )}. 
Pollack  \cite{pollack} showed it in the $1$-variable case when $R$ is a Noetherian integral domain with infinitely many maximal ideals, or, if $R=S[u]$ with $S$ an integral domain. His method relies on a clever use of the Eisenstein criterion.

Finding Goldbach decompositions for ${\mathcal Q}\in R[\underline x]$ ($n\geq 1$) corresponds to the special situation of the degree $1$ case of the Schinzel hypothesis for which $\underline P=\{P_1,P_2\}$ with $P_1=-y$ and $P_2=y+{\mathcal Q}$. We obtain this result.

\begin{corollary} \label{cor:goldbach}
Let $R$ be a ring as in Theorem \ref{thm:schinzel1}. Every nonconstant poly\-nomial ${\mathcal Q}\in R[\underline x]$ is the sum of two irreducible polynomials $F,G\in R[\underline x]$ with $F=a + b \hskip 1pt x_1^{d_1} \cdots x_n^{d_n}$ 
{\rm (}$a,b\in R${\rm )} a binomial of degree $d_1+\cdots+d_n \leq \deg({\mathcal Q})$.
\end{corollary}

One can even take $d_1+\cdots+d_n=1$ when $R=K$ is a Hilbertian field, or when $n\geq 2$ and $R=K$ is an infinite field (the latter was already known from \cite[Corollary 4.3(2)]{BDN1}). On the other hand, the Goldbach conjecture fails for $\Ff_2[x]$ and ${\mathcal Q}(x)=x^2+x$ (note that $x^2+x+1$ is the only irreducible polynomial in $\Ff_2[x]$ of degree $2$). From Corollary \ref{cor:goldbach} however, it holds true for $\Ff_q[x,y]$ if condition $\deg(F)\leq \deg({\mathcal Q})$ is replaced by $\deg_{x}(F)\leq \deg_{x}({\mathcal Q})$.

\subsection{Spectra} The following result uses Theorem \ref{thm:Main} as a main ingredient.
\begin{corollary} \label{cor:spectre}
Assume that $n\geq 2$ and $k$ is an arbitrary field. 
Let ${\mathcal S}\subset k$ be a finite subset, $a_0\in \overline k\setminus {\mathcal S}$, separable over $k$ 
and $V\in k[\underline x]$ a nonzero polynomial. Then, for all suitably large integers $d_1,\ldots,d_n$ {\rm (}larger than some $d_0$ depending on ${\mathcal S}$, $a_0$, $V${\rm )}, there is a polynomial $U\in k[\underline x]$
such that:
\vskip 0,5mm

\noindent
\hskip -3mm
$\begin{matrix}
& \hbox{{\rm (a)} $U(\underline x) - a V(\underline x)$ is reducible in $k[\underline x]$ for every $a \in {\mathcal S}$,} \hfill\\
& \hbox{{\rm (b)} $U(\underline x) - a_0 V(\underline x)$ is irreducible in $k(a_0)[\underline x]$ of degree $\max(\deg(U),\deg(V))$,}  
\hfill\\
& \hbox{{\rm (c)} $\deg_{x_i}(U)=d_i$, $i=1,\ldots,n$.} \hfill \\
\end{matrix}$
\end{corollary}

If ${\mathcal S}\not=k$, e.g. if $k$ is infinite, $a_0$ can be chosen in $k$ itself.

A more precise version of Corollary \ref{cor:spectre}, given in \S \ref{sec:spectra}, shows that one can even prescribe all irreducible factors  but one of each polynomial $U(\underline x)-a V(\underline x)$, $a\in {\mathcal S}$, provided that these factors satisfy some standard condition.

If $k$ is algebraically closed, the irreducibility condition (b) implies that the rational function $U/V$ is {\it indecomposable} \cite[Theorem 2.2]{bodin_isrJ}; ``indecompos\-able'' means that $U/V$ cannot be written $h\circ H$ with $h\in k(u)$ and $H\in k(\underline x)$ with $\deg(h)\geq 2$. The set of all $a\in k$ such that $U(\underline x) - a V(\underline x)$ is reducible in $k[\underline x]$ is called the {\it spectrum} of $U/V$ and the indecomposability condition equivalent to the {spectrum} being finite.
Corollary \ref{cor:spectre} rephrases to conclude that 
given ${\mathcal S}$ and $V$ as above, indecomposable rational functions $U/V\in k(\underline x)$ 
exist with a spectrum containing ${\mathcal S}$ and satisfying (c). See 
\cite{salah1} \cite{salah2} for the special case $V=1$
and \cite[\S 3.1.1]{BDN3} for further results.

\subsection{Hilbertian rings} \label{ssec:hilbertian_rings}
Except for Theorem \ref{thm:Main2} for which we use geometrical tools (\S \ref{sec:ngeq2}),
we follow a Hilbert like specialization approach.

Given an irreducible polynomial $F(\underline \lambda, \underline x) \in R[\underline \lambda, \underline x]$ with $\deg_{\underline x}(F)\geq 1$, the 
Hilbert 
property provides specializations $\lambda_1^\ast,\ldots, \lambda_r^\ast\in K$ of the indeterminates from $\underline \lambda$ such that 
$F(\lambda_1^\ast,\ldots, \lambda_r^\ast,\underline x)$ is irreducible in $K[\underline x]$ (\S \ref{ssec:reminder}). 

As suggested above and detailed in \S \ref{sec:preliminary}, the challenge for our purpose is to make it work {\it over the ring $R$}, i.e., to be able to find $\lambda_1^\ast,\ldots, \lambda_r^\ast$ in $R$ such that $F(\lambda_1^\ast,\ldots, \lambda_r^\ast,\underline x)$ is irreducible in $R[\underline x]$. A problem however is that this is false in general, even with $R=\Zz$. Take $F=(\lambda^2-\lambda) \hskip 2pt x + (\lambda^2-\lambda+2)$ in $\Zz[\lambda,x]$; for every $\lambda^\ast\in \Zz$, $F(\lambda^\ast, \underline x)$ is divisible by $2$, hence reducible in $\Zz[x]$. 

To remedy this problem, we develop the notion of {\it Hilbertian ring} introduced in \cite[\S 13.4]{FrJa}. 
The defining property is that, for separable polynomials $F(\underline \lambda,x)$ in the one variable $x$, tuples $(\lambda_1^\ast,\ldots, \lambda_r^\ast)$ can be found with coordinates in the ring $R$ and satisfying the specialization property over $K$.  

Our approach 
to reach irreducibility over $R$
can be summarized as follows.
It may be of interest for the sole sake of the Hilbertian field theory.
\vskip 1,5mm

\noindent ({\it Hilbert sections \ref{sec:Hilbert} and \ref{sec:Schinzel}})
Assume that $K$ is of characteristic $0$, 
or $K$ is of characteristic $p>0$ and imperfect (the {\it imperfectness assumption}).

\noindent
\hskip 2mm {\rm (a)} 
We extend the property of Hilbertian rings to all irreducible polynomials $F(\underline \lambda, \underline x)$ (not just the separable ones $F(\underline \lambda, x)$), and show in fact a stronger version: $\lambda_1^\ast,\ldots, \lambda_r^\ast$ can be chosen pairwise relatively prime (Prop.\ref{prop:I-HilbertUFD}); and for $R=k[u]$,
their degrees in $u$ can be prescribed \hbox{off a finite range \hskip -2pt (Theorem \ref{lemma:hilbert2}).}

\noindent
\hskip 2mm {\rm (b)} We show that if $K$ is a field with the product formula, then $R$ is a Hil\-ber\-tian ring (Theorem  \ref{lemma:hilbert1}); this improves on \cite[Prop.13.4.1]{FrJa} where the as\-sumption is that $R$ is finitely generated over $\Zz$, or over $k[u]$ for some\hbox{ field $k$.}

\noindent
\hskip 2mm {\rm (c)} For $R$ both a UFD and a Hilbertian ring, we show that our polynomials $F(\underline \lambda, \underline x)$, due to their structure, satisfy the specialization property {\it over the ring $R$},
and we prove Theorem \ref{thm:schinzel1} in this situation (\S  \ref{sec:Schinzel}). 
\vskip 1mm

The imperfectness assumption relates to a classical subtlety in positive cha\-rac\-te\-ris\-tic. There are two notions of {\it Hilbertian fields}, depending on whether the specialization property is requested for all irreducible polynomials or only for the separable ones.
We follow \cite{FrJa} and use the name {\it Hilbertian} for the weaker (the latter), and we say {\it totally Hilbertian} for the stronger (precise definitions are in \S \ref{ssec:reminder}). They are equivalent under the {imperfectness assumption} \cite{uchida} \cite[Proposition 12.4.3]{FrJa}.

\vskip 2mm

\noindent
{\bf Final note.} 
The original Schinzel hypothesis has also appeared in \hbox{Arithmetic} Geometry, notably around the question of whether, for appropriate varieties over a number field $k$, the Brauer-Manin 
obstruction is the only obstruction to the Hasse principle: if rational points exist locally (over all completions of $k$), they should exist globally (over $k$). In 1979, Colliot-Th\'el\`ene and Sansuc \cite{colliot-sansuc} noticed that this is true for a large family of conic bundle surfaces over $\Pp^1_\Qq$ if one assumes Schinzel's hypothesis. This conjectural statement has become since a working hypothesis of the area. See for example \cite{harpaz-wittenberg2016} for some last developments. Although the number field environment seems closely tied to the question, it could be interesting to investigate the potential use of our polynomial version of the Schinzel hypothesis to some similar questions over appropriate fields like rational function fields.

\vskip 1,5mm
The paper is organized as follows. The strategy is detailed in \S \ref{sec:preliminary}. \S \ref{sec:ngeq2} is devoted to the geometric case that $R=k[\underline x]$ with $n\geq 2$ and $k$ is an infinite field; Theorem \ref{thm:Main2} is proved. \S \ref{sec:Hilbert} is the Hilbert part. The main results from \S \ref{sec:intro} (other than Theorem \ref{thm:Main2}) are finally proved in \S \ref{sec:Schinzel}.

\section{General strategy} \label{sec:preliminary}

Throughout the paper, $R$ is a UFD with fraction field $K$. Recall that a polynomial with coefficients 
in $R$ is said to be {\it primitive \hbox{w.r.t.} $R$} if its coefficients are relatively prime  in $R$. 

All indeterminates are algebraically 
independent over $\overline K$.

Let $\underline x = (x_1,\ldots,x_n)$ ($n\geq 1$) and $\lambdabar=(\lambda_0, \lambda_1,\ldots,\lambda_\ell)$ 
($\ell \geq 1$) be two tuples of indeterminates and let ${\underline Q}= (Q_0, Q_1,\ldots,Q_\ell$) with $Q_0=1$
be a $(\ell+1)$-tuple of nonzero polynomials in $R[\underline x]$, distinct up to multiplicative constants 
in $K^\times$. Set
\vskip 1mm

\centerline{$M(\lambdabar,\xbar) = \sum_{i=0}^{\ell} \lambda_i \hskip 1pt Q_i(\underline x)$.} 

\vskip 1mm

Consider a set $\underline{P}=\{P_1,\ldots, P_s\}$ of $s$ polynomials 
\vskip 1mm

\centerline{$P_i(\underline x,y)= P_{i\degP_i}(\underline x) \hskip 2pt y^{\degP_i} +\cdots+ P_{i1}(\underline x) \hskip 1pt y + P_{i0}(\underline x)$,}
\vskip 1mm

\noindent
irreducible in $R[\underline x,y]$ and of degree $\degP_i \geq 1$ in $y$, $i=1,\ldots,s$.
Each polynomial $P_i(\underline x,y)$ is irreducible in $K(\underline x)[y]$ and is
primitive \hbox{w.r.t.} $R[\underline x]$. 

Finally set, for $i=1,\ldots,s$,
 
\vskip 1,5mm 
\centerline{$F_i(\lambdabar,\xbar) 
= P_i(\underline x, M(\lambdabar,\xbar))= P_i(\underline x, \sum_{j=0}^\ell \lambda_j Q_j(\underline x))$.}
\vskip 1,5mm

\noindent
In the case $\rho_i = 1$, i.e., $P_i = A_i(\underline x) + B_i(\underline x)\hskip 1pt  y$, the polynomial $F_i$ rewrites

\vskip 1mm
\centerline{$ F_i(\lambdabar,\xbar)= A_i(\xbar) + B_i(\xbar) \left(\sum_{j=0}^\ell \lambda_j Q_j(\underline x))\right)$.}
\vskip 1mm

\begin{lemma} \label{lemma:first}
{\rm (a)} Each polynomial $F_i(\lambdabar,\xbar)$ is irreducible in $R[\underline \lambda,\xbar]$ and of de\-gree $\geq 1$ in $\underline x$. Furthermore, 
if $\deg_y(P_i)=1$, $F_i(\lambdabar,\xbar)$ \hbox{is irreducible
in $\overline K[\lambdabar,\xbar]$.}
\vskip 1,2mm

\noindent
{\rm (b)} If $R$ is infinite and $\Pi= \prod_{i=1}^s P_i$, there is no irreducible polynomial $p \in R[\underline x]$ dividing all polynomials $\Pi(\underline x, M(\underline x))$ with $M\in R[\underline x]$.
\end{lemma}

Note that (b) fails if $R$ is finite: with $R=\Ff_2$ and $\underline P=\{x,x+1\}$, the polynomial $x$ divides all polynomials $M(x) (M(x)+1)$ ($M\in \Ff_2[x]$).

\begin{proof}
(a) Fix an integer $i\in \{1,\ldots,s\}$. By assumption, the polynomial $P_i(\underline x, \lambda_0)$ is irreducible in $R[\underline x,\lambda_0]$. It is also irreducible in the bigger ring $R[\underline x,\underline \lambda]$.  Consider the ring automorphism $R[\underline x,\underline \lambda] \rightarrow R[\underline x,\underline \lambda]$ that is the identity on $R[\underline x, \lambda_1,\ldots,\lambda_{\ell}]$ and maps $\lambda_0$ to 
the polynomial $\lambda_0 +  \sum_{i=1}^{\ell} \lambda_i Q_i (\underline x)$. The polynomial $F_i(\lambdabar,\xbar)$ is the image of $P_i(\underline x, \lambda_0)$ by this isomorphism. Hence it is irreducible in $R[\underline x,\underline \lambda]$.

To see that $\deg_{\underline x}(F_i) \geq 1$, write $F_i$ as a polynomial in $\lambda_1$. The leading coefficient is $P_{i\rho_i}(\underline x) Q_1(\underline x)^{\rho_i}$; it is of positive degree in $\underline x$ since $Q_1$ is by assumption. This proves that 
$\deg_{\underline x}(F_i) \geq 1$.

In the case $\degP_i=1$, irreducibility of $F_i(\lambdabar,\xbar)$ in $\overline K[\underline x,\underline \lambda]$ follows from the above case, applied with $R$ taken to be $\overline K$, and the fact that the polynomial $P_i(\underline x,y)=A_i(\underline x) + B_i(\underline x)\hskip 1pt y$ is irreducible in $\overline K[\underline x,y]$. Namely $P_i(\underline x,y)$ is of degree $1$ in $y$ and is primitive \hbox{w.r.t.} $\overline K[\underline x]$. Primitivity follows from the fact that, as $A_i$ and $B_i$ are relatively prime in $R[\underline x]$, then 

\noindent 
- they are relatively prime in $K[\underline x]$ (an application of Gauss's lemma), and,

\noindent
- they are relatively prime in $\overline K[\underline x]$. 
For lack of reference for this last point, we provide below a quick argument. 

Prove by induction on $n$ that for every field $K$, for every nonzero $A,B\in K[\underline x]$, if $A$ and $B$ have a common divisor $D\in \overline K[\underline x]$ with $\deg(D)>0$, they have a common divisor $C\in K[\underline x]$ with $\deg(C)>0$. The case $n=1$  follows from the B\'ezout theorem. Then, for $n\geq 2$, if $D$ is as in the claim, we may assume that $\deg_{(x_2,\ldots,x_n)}(D)>0$. Observe then that $D$ divides $A$ and $B$ in $\overline{K(x_1)}[x_2,\ldots,x_n]$. By induction $A$ and $B$ have a common divisor $C\in K(x_1)[x_2,\ldots,x_n]$ with $\deg_{(x_2,\ldots,x_n)}(C)>0$. Using Gauss's lemma, one easily constructs a polynomial $C_0 = c(x_1) C \in 
  K[x_1][x_2,\ldots,x_n]$ (with $c(x_1)\in K[x_1]$) dividing both $A$ and $B$ in $K[x_1][x_2,\ldots,x_n]$.
 \vskip 1mm
  
 \noindent
 (b) If the claim is false, there is an irreducible polynomial $p \in R[\underline x]$ such that
 $\Pi(\underline x, M(\underline x)) = 0$ in the quotient ring $R[\underline x]/(p(\underline x))$ for all $M\in R[\underline x]$.
 But $R[\underline x]/(p(\underline x))$ is an integral domain, and it is infinite. Indeed, if $p$ is nonconstant, say $d=\deg_{x_1}(p)\geq 1$, the elements $\sum_{i=0}^{d-1} r_i x_1^i$ with $r_0,\ldots,r_{d-1}\in R$ are infinitely many different elements in
 $R[\underline x]/(p(\underline x))$; and if $p\in R$, then the quotient ring 
 is $R/(p)[\underline x]$, which is infinite too. Conclude that the polynomial $\Pi(\underline x, y)$ which has infinitely many 
 roots in $R[\underline x]/(p(\underline x))$ is zero in the ring $(R[\underline x]/(p(\underline x))[y]$. As this ring is an integral domain, there is an index $i\in\{1,\ldots,s\}$ such that $P_i(\underline x,y)$ is zero in $(R[\underline x]/(p(\underline x))[y]$. This contradicts $P_i(\underline x, y)$ being primitive w.r.t. $R[\underline x]$.
\end{proof}

Denote the set of polynomials $F_1,\ldots,F_s$ by ${\underline F}$
and consider the subset 
\vskip 1mm

\centerline{${H}_{R}({\underline F})  \subset R^{{\ell}+1}$,} 
\vskip 1mm

\noindent
of all $(\ell+1)$-tuples $\underline \lambda^\ast = (\lambda_0^\ast,\ldots,\lambda_\ell^\ast)$ such that 
$F_i(\underline \lambda^\ast, \underline x)$ is irreducible in $R[\underline x]$, for each $i=1,\ldots,s$. It can be equivalently viewed as the set of all 
polynomials of the form $m(\underline x) = \sum_{j=0}^\ell m_j Q_j(\underline x)$ ($m_0,\ldots,m_\ell \in R$) such that 
$P_i(\underline x, m(\underline x))$ is irreducible in $R[\underline x]$, $i=1,\ldots,s$.

Theorems \ref{thm:schinzel1} -- \ref{thm:Main}
will be obtained {\it via}
the following special case of our situation: for a given $\underline d=(d_1,\ldots,d_n) \in (\Nn^\ast)^n$,  the polynomials $Q_i$ are 
all the monic monomials $Q_0,Q_1,\ldots,Q_{N_{\underline d}}$ in ${\mathcal Pol}_{R,n,\underline d}$. The polynomial
\vskip 1mm

\centerline{$M_{\underline d}(\lambdabar,\xbar) = \sum_{i=0}^{N_{\underline d}} \lambda_i \hskip 1pt Q_i(\underline x)$} 

\vskip 1mm

\noindent
is then the {\it generic polynomial in $n$ variables of $i$-th partial degree $d_i$, $i=1,\ldots,n$}, and 
Theorems \ref{thm:schinzel2} and \ref{thm:Main} are about the set
\vskip 1mm  
  
\centerline{${H}_{R}({\underline F}) = {\mathcal Irr}_{n}(R,{\underline P})\cap {\mathcal Pol}_{R,n,{\underline d}}$ }
\vskip 1mm

For example, anticipating on the reminder on Hilbertian fields in \S \ref{ssec:reminder}, 
we can immediately prove this statement, already alluded to in \S \ref{sec:intro}.
\vskip 2,5mm

\noindent
{\bf Addendum to Theorem \ref{thm:schinzel1}.}  {\it The set ${\mathcal Irr}_{n}(R,\underline P)$ is Zariski-dense in ${\mathcal Pol}_{R,n,\underline d}$ for every $\underline d\in (\Nn^\ast)^n$, in each of these two situations:
\vskip 0,5mm

\noindent
{\rm (a)} $R=K$ is a totally Hilbertian field, 
\vskip 0,5mm

\noindent
{\rm (b)} $R=K$ is a Hilbertian field and $\deg_y(P_1) = \ldots = \deg_y(P_s) = 1$. 
}
\vskip 2,5mm

\begin{proof} By definition, ${H}_{K}({\underline F})$ is a {\it Hilbert subset}. Furthermore, from Remark \ref{rem:degree1_Hilbertian}, it contains a {\it separable Hilbert subset} if $\deg_y(P_1) = \ldots = \deg_y(P_s) = 1$. It follows from the definitions that ${H}_{K}({\underline F})$ is Zariski-dense in $K^{N_{\underline d}+1} = {\mathcal Pol}_{K,n,{\underline d}}$ in both situations. One does not even need to assume that  $\displaystyle d_1+\cdots+d_n \geq \max_{1\leq i\leq s} \deg_{\underline x}(P_i)+2$; the statement holds for example \hbox{for $d_1=\ldots = d_n=1$.}\end{proof}
\vskip 1mm

When $R$ is more generally a ring, we have to further guarantee that: 

\noindent
- the Hilbert subset ${H}_{K}({\underline F})$ contains $(\ell +1)$-tuples with coordinates in $R$,

\noindent
- for some of these $(\ell +1)$-tuples $\underline \lambda^\ast$, the corresponding polynomials $F_i(\lambdabar^\ast,\xbar)$ are primitive  \hbox{w.r.t.} $R$, and so irreducible in $R[\underline x]$.
 
For $R=k[u_1,\ldots,u_r]$, polynomials in $R[\underline x]$ can be viewed as polynomials in at least two variables over the field $k$. 
We explain in  \S \ref{sec:ngeq2} how geometric specialization techniques can be used, if $k$ is also infinite.
For more general rings $R$, more arithmetic specialization tools are needed, which we develop in \S \ref{sec:Hilbert}. The specific argument for the primitivity point is given in \S \ref{ssec:schinzel_proof}; it takes advantage of the special form of the polynomial $F_i$ and, as mentioned before, cannot extend  to arbitrary polynomials $F \in R[\lambdabar,\xbar]$.

\section{The geometric part} \label{sec:ngeq2}

Lemma \ref{lemma:general2} is our specialization tool here. Based on results of Bertini, Krull and Noether, it is in the same vein as those from \cite{BDN1}, \cite{BDN3}.
We prove it in \S \ref{ssec:general}, then deduce Theorem \ref{thm:Main2} in \S \ref{ssec:proof-Dirichlet}.

\subsection{The specialization lemma} \label{ssec:general} Notation is as in \S \ref{sec:preliminary}.  Consider the special case of the general situation from \S \ref{sec:preliminary} for which $s=1=\rho_1$. One degree $1$ polynomial $P(\underline x,y)$ is given: $P(\underline x,y) = A(\underline x) + B(\underline x) y$ with $A, B\in R[\underline x]$ two nonzero relatively prime polynomials,
or $P(\underline x,y)=y$. 
We then have:

\vskip 1mm
\centerline{$\begin{matrix}
F(\lambdabar,\xbar) & = A(\xbar) + B(\xbar) \left(\sum_{j=0}^\ell \lambda_j Q_j(\underline x)\right)\hfill \\
& = A(\xbar) + \lambda_0 B(\xbar)+ \lambda_1 B(\xbar) Q_1(\xbar) + \cdots
  +\lambda_\ell B(\xbar) Q_\ell(\xbar)
\end{matrix}$}
\vskip 1mm

\begin{lemma} \label{lemma:general2} Assume that $n\geq 2$, $R=K$ is an algebraically closed field
and the following holds {\rm (}which implies $\ell \geq 1${\rm )}:
\vskip 0,2mm

{\rm (a)} there is an index $i_0\in \{1,\ldots,\ell\}$ such that 

\hskip 3mm - $\deg(Q_{i_0})\not\equiv 0$ modulo $p$ if $\charac(K) = p >0$, 

\hskip 3mm - $\deg(Q_{i_0})\not=0$ if $\charac(K) = 0$.

\vskip 0,5mm

{\rm (b)} there is \underbar{no} polynomial $\chi \in K[\underline x]$ such that $A,B,Q_1,\ldots,Q_\ell \in K[\chi]$.
\vskip 0,6mm

\noindent
Then the set ${H}_{K}({F})$ of all $(\ell+1)$-tuples $\underline \lambda^\ast = (\lambda_0^\ast,\ldots,\lambda_\ell^\ast)$ such that 
$F(\underline \lambda^\ast, \underline x)$ is irreducible in $K[\underline x]$  contains a nonempty Zariski 
open subset of $K^{\ell+1}$.
\end{lemma}

\begin{remark} Assumptions (a) and (b) can probably be improved but the following examples show they cannot be totally removed. In each of them, $F(\lambdabar,\xbar)$ is reducible in $\overline{K(\lambdabar)} [\underline x]$ and every non-trivial factorization yields a Zariski dense subset of $\lambdabar\point \in K^{\ell+1}$ such that $F(\lambdabar\point,\xbar)$ is reducible in $K[\xbar]$. 
\vskip 1mm

\noindent
$\bullet$ If $A,B,Q_1,\ldots,Q_\ell \in K[\chi]$ for some $\chi \in K[\underline x]$, one can write 
$F(\lambdabar,\xbar) = h(\chi)$ with $h\in \overline{K(\lambdabar)} [u]$. If $\deg(h)\geq 2$, $h$ is reducible and so is 
$F(\lambdabar,\xbar)$ in $\overline{K(\lambdabar)} [\underline x]$.

\vskip 1mm

\noindent
$\bullet$ For $A = x_1^2$, $B = -x_2^2$, $\ell = 1$ and $Q_0= Q_1=1$, we have
\vskip 0,5mm

\centerline{$F(\lambdabar,\xbar) = x_1^2 - \lambda_0 x_2^{2} -\lambda_1 x_2^2= (x_1 - \sqrt{\lambda_0+\lambda_1} \hskip 1,5pt x_2)(x_1+\sqrt{\lambda_0+\lambda_1} \hskip 1,5pt x_2)$.}
\vskip 0,5mm

\vskip 0,5mm

\noindent
$\bullet$ If ${\rm char}(K)=p>0$, for $A = x_1^p$, $B = x_2^p$, $\ell = 1$, $Q_0=1$, $Q_1=x_2^p$, we have
\vskip 0,5mm

\centerline{$F(\lambdabar,\xbar) = x_1^p + \lambda_0 x_2^{p} + \lambda_1 x_2^{2p}= (x_1 + \lambda_0^{1/p} x_2 + \lambda_1^{1/p} x_2^{2})^p$.}
\vskip 0,5mm
\end{remark}

\begin{proof}
Assume that the conclusion of Lemma \ref{lemma:general2} is false.
From the Bertini-Noether theorem \cite[Prop.~9.4.3]{FrJa}, $F(\lambdabar,\xbar)$ is reducible in $\overline{K(\lambdabar)} [\underline x]$. Clearly then polynomials $F(\underline x, \underline \lambda^\ast)$ are reducible in $K[\underline x]$ for all $\underline \lambda^\ast \in K^{\ell+1}$ such that $\deg(F(\underline x, \underline \lambda^\ast)) = \deg_{\underline x}(F)$.
The Bertini-Krull theorem \cite[Theorem 37]{schinzel-book2000} then yields that one 
of the following conditions holds:

\noindent
\begin{enumerate}
  \item \label{iit:charp} $\charac(K) = p >0$ and
    $F(\lambdabar, \xbar) \in K[\lambdabar,\xbar^p]$ with $\xbar^p =
    (x_1^p,\ldots,x_n^p)$,
  \item \label{iit:hom} there exist $\phi,\psi \in K[\xbar]$ with
    $\deg_{\xbar}(F) > \max(\deg(\phi), \deg(\psi))$ satisfying the following:
    there is an integer $\delta\geq 1$  and $\ell +2$ polynomials $H, H_0,H_1,\ldots,H_\ell \in K[u,v]$
    homogeneous of degree $\delta$ such that
$$\left\lbrace
\begin{array}{c}
A(\xbar) = H(\phi(\xbar),\psi(\xbar)) = \sum_{i=0}^{\delta} h_{i} \hskip 2pt
\phi(\xbar)^i \psi(\xbar)^{\delta-i} \hfill \\
B(\xbar) = H_0(\phi(\xbar),\psi(\xbar)) =  \sum_{i=0}^{\delta} h_{0i} \hskip 2pt
\phi(\xbar)^i \psi(\xbar)^{\delta-i}\ \hfill \\
BQ_1(\xbar) = H_1(\phi(\xbar),\psi(\xbar)) =  \sum_{i=0}^{\delta} h_{1i} \hskip 2pt
\phi(\xbar)^i \psi(\xbar)^{\delta-i}\ \hfill \\
\ \vdots \ \\
BQ_\ell(\xbar) = H_\ell(\phi(\xbar),\psi(\xbar)) = \sum_{i=0}^{\delta}
h_{\ell i} \hskip 2pt\phi(\xbar)^i \psi(\xbar)^{\delta-i}\\
\end{array}
\right.$$
\end{enumerate}

\noindent
The rest of the proof consists in ruling out both conditions (1) and (2).

For condition (1), this readily follows from the assumption on $\deg(Q_{i_0})$: if $\charac(k) = p >0$, the polynomials $B$ and $BQ_{i_0}$ cannot be both in $K[\underline x^p]$.

Assume condition (2) holds. Note that the polynomials $\phi$ and $\psi$ are relatively prime in 
$K[\underline x]$ as a consequence of $A,B$ being relatively prime in 
$K[\underline x]$. We claim that the two conditions
\vskip 1mm

\centerline{$\left\lbrace
\begin{array}{c}
B(\xbar) = H_0(\phi(\xbar),\psi(\xbar))\ \hfill \\
BQ_{i_0}(\xbar) = H_{i_0}(\phi(\xbar),\psi(\xbar)) \hfill \\
\end{array}
\right.$}
\vskip 1mm

\noindent
lead to this conclusion: there is $(\beta,\gamma)\in K^2$ such that $\beta 
  \phi(\underline x) +\gamma \psi(\underline x) = 1$. We show it by induction on the common degree $\delta$ of 
  $H_0$ and $H_{i_0}$.

For $\delta=1$, write $B= a \phi+b \psi$ and $BQ_{i_0} = a^\prime \phi+b^\prime \psi$ with $a,b,a^\prime,b^\prime
  \in K$. If $\deg(B)=0$, then $a \phi+b \psi \in K\setminus\{0\}$ and the claim is established. Assume $\deg(B)>0$. If $ab^\prime - a^\prime b\not=0$, any irreducible factor $\pi$ of $B$ divides 
$a\phi+ b\psi$ and $a^\prime \phi + b^\prime \psi$,
hence divides both
$\phi$ and $\psi$ in $K[\underline x]$, which contradicts $\phi$ and $\psi$ being relatively prime. As there is at least one such factor $\pi$, we have
 $(a,b)= \kappa (a^\prime,b^\prime)$ for some nonzero $\kappa \in K$. It follows that $B= \kappa BQ_{i_0}$ and $\deg(Q_{i_0})=0$. 
 This contradicts our assumption. Hence the claim is established for $\delta=1$.

Assume the claim is proved for $\delta\geq 1$ and that 
\vskip 1mm

\centerline{$\left\{ \begin{matrix}
B= \prod_{j=1}^{\delta+1} (a_j\phi + b_j \psi) \\
BQ_{i_0}= \prod_{j=1}^{\delta+1} (a_j^\prime\phi + b_j^\prime \psi)\\
  \end{matrix} \right.$}
 \vskip 1mm
 
\noindent
for some $(\delta+1)$-tuples $((a_1,b_1),\ldots,(a_{\delta+1},b_{\delta+1}))$ and $((a^\prime_1,b^\prime_1),\ldots,(a^\prime_{\delta+1},b^\prime_{\delta+1}))$ with components in $K^2$.

If $\deg(B)=0$, all polynomials $a_j\phi + b_j \psi$, $j=1,\ldots,\delta+1$, are of degree $0$. Hence there exists $(\beta,\gamma)\in K^2$ such that $\beta 
  \phi +\gamma \psi = 1$. Assume $\deg(B)>0$. As above in the case $\delta=1$, use an irreducible factor of $B$ in 
  $K[\underline x]$ to conclude that there exist two indices $j$, $j^\prime$ such that 
 this irreducible factor divides both $a_j\phi + b_j \psi$ and $a_{j^\prime}^\prime \phi + b_{j^\prime}^\prime \psi$. We may assume that 
 $j=j^\prime = \delta+1$. As above in the case $\delta=1$, it follows from $\phi$, $\psi$ relatively prime in $K[\underline x]$
 that  \vskip 1mm
 
 \centerline{$a_{\delta+1}\phi + b_{\delta+1}\psi = \kappa (a_{{\delta+1}}^\prime \phi + b_{{\delta+1}}^\prime \psi)$}
 \vskip 1mm

 \noindent
 for some nonzero $\kappa \in K$. Consider the polynomial $B_1 = B/(a_{\delta+1}\phi + b_{\delta+1} \psi)$. 
 It is nonzero and we have 
 \vskip 1mm

\centerline{$\left\{ \begin{matrix}
B_1= \prod_{j=1}^{\delta} (a_j\phi + b_j \psi) \\
\kappa B_1Q_{i_0}= \prod_{j=1}^{\delta} (a_j^\prime\phi + b_j^\prime \psi)\\
  \end{matrix} \right.$}
\vskip 1mm

\noindent
From the induction hypothesis, applied to $B_1$ and $\kappa B_1Q_{i_0}$, there is $(\beta,\gamma)\in K^2$ such that $\beta 
  \phi +\gamma \psi = 1$. This completes the proof of our claim.

\noindent

Fix $(\beta, \gamma)\in K^2$ such that $\beta \phi +\gamma \psi = 1$. Pick $(a,b) \in K^2$ such that 
$a\gamma - \beta b\not= 0$
and set $\chi = a\phi + b\psi$. 
We have $\deg(\chi) >0$. 
Then $K \phi + K \psi = K \chi + K$ and so $A,B,BQ_1,\ldots,BQ_\ell$ are in $K[\chi]$. It follows that $A,B,Q_1,\ldots,Q_\ell$ are in $K[\chi]$ too. Here is an argument. Fix $i\in \{1,\ldots,\ell\}$. Since $B, BQ_i\in K[\chi]$, $Q_i$ writes $Q_i=(p/q)(\chi)$ for some $p,q\in K[t]$ relatively prime. But then there exists $u,v\in K[t]$ such that $u(\chi) p(\chi)+v(\chi)q(\chi)=1$. Since $q(\chi)$ divides $p(\chi)$ in $K[\underline x]$, we have $\deg(q)=0$.
Hence $Q_i\in K[\chi]$.
\end{proof}

\subsection{Proof of Theorem \ref{thm:Main2}} \label{ssec:proof-Dirichlet}
Assume $n\geq 2$, fix an infinite field $k$, two nonzero relatively prime polynomials $A$, $B$ in $k[\underline x]$ and a $n$-tuple $\underline d \in (\Nn^\ast)^n$. As explained in \S \ref{sec:preliminary}, consider the special case of Lemma \ref{lemma:general2}  for which the polynomials $Q_i$ are all the monomials $Q_0,\ldots,Q_{N_{\underline d}}$ in ${\mathcal Pol}_{k,n,\underline d}$ (with $Q_0=1$). We then have $F(\lambdabar,\xbar)=  A(\xbar) + B(\xbar) M_{\underline d}(\underline \lambda,\xbar)$ with $M_{\underline d} = \sum_{i=0}^{N_{\underline d}} \lambda_i Q_i $ the generic polynomial in $n$ variables of partial degree $d_i$  in $x_i$, $i=1,\ldots,n$.

Lemma \ref{lemma:general2} concludes that ${H}_{k}({F}) = {\mathcal Irr}_{n}(\overline k,A,B)\cap {\mathcal Pol}_{k,n,{\underline d}}(\overline k)$ contains a nonempty Zariski open subset of ${\mathcal Pol}_{k,n,{\underline d}}(\overline k)$. 
As $k$ is infinite, the set ${\mathcal Irr}_{n}(\overline k,A,B)\cap {\mathcal Pol}_{k,n,{\underline d}}(k)$ also contains a nonempty Zariski open subset of ${\mathcal Pol}_{k,n,{\underline d}}(k)$. This proves Theorem \ref{thm:Main2}.

\begin{remark}
(a) If $k$ is finite however, non emptyness of ${\mathcal Irr}_{n}(k,A,B)$ cannot be guaranteed at this stage: each finite set ${\mathcal Irr}_{n}(k,A,B) \cap {\mathcal Pol}_{k,n,\underline d}(k)$ ($\underline d \in (\Nn^\ast)^n$) could be covered by an hypersurface.  
For infinite fields, Theorem \ref{thm:Main2} clearly covers Theorem \ref{thm:Main}. We will use a different method, in \S  \ref{sec:Hilbert}, to prove Theorem \ref{thm:Main} for finite fields (which will also reprove the infinite case).
\vskip 0,5mm

\noindent
(b) Lemma \ref{lemma:general2} can be used in other situations. For exam\-ple, let $A,B,C\in K[\underline x]$ be nonzero polynomials, with $A$, $B$ relatively prime and $C\in K[\underline x]$ distinct from $A$, $B$, up to multiplicative constants in $K^\times$. Assume hypotheses (a) and (b) of Lemma \ref{lemma:general2} respectively hold for $Q_{i_0} = C$ and for $A,B,C$. Lemma \ref{lemma:general2} shows that the set of $(\lambda, \mu) \in K^2$ such that $A+B(\lambda C+ \mu)$ is irreducible in $K[\underline x]$ contains a nonempty Zariski open subset of $\Aa^{2}_{K}$.
\end{remark}

\section{The Hilbert side} \label{sec:Hilbert} 

This section introduces the notion of Hilbertian ring and establishes some corresponding 
specialization tools, which will be important ingredients of the proofs of the main theorems in \S \ref{sec:Schinzel}.

\subsection{Basics from the Hilbertian field theory} \label{ssec:reminder}
We recall the basic definitions and refer to chapters 12 and 13 of \cite{FrJa} for more. Other classical references include 
\cite{schinzel-book}, \cite{schinzel-book2000}, \cite{langFG}.

Consider a field $K$ and two tuples $\underline \lambda = (\lambda_1,\ldots, \lambda_r)$ and $\underline x= (x_1,\ldots,x_n)$ ($r\geq 1$, $n\geq 1$) of indeterminates. Given $m$ polynomials $f_1(\underline \lambda,\underline x), \ldots, f_m(\underline \lambda,\underline x)$ ($m\geq 1$) in $\underline x$ with coefficients in $K(\underline \lambda)$, irreducible in the ring $K(\underline \lambda)[\underline x]$
and a polynomial $g\in K[\underline \lambda]$, $g\not= 0$,
consider the set

\vskip 1,5mm

\centerline{$\displaystyle H_K(f_1,\ldots,f_m;g)=\left\{ \underline \lambda^\ast \in K^r \hskip 1mm \left | 
\begin{matrix}
\hskip 1mm f_i( \underline \lambda^\ast,  \underline x)\hskip 1mm  \hbox{\rm irreducible in } K[\underline x] \cr
\hskip 1mm \hbox{\rm for each }i=1,\ldots,m,  \hfill \cr
\hskip 1mm \hbox{\rm and}\ g(\underline \lambda^\ast) \not=0. \hfill \cr
\end{matrix} \right.
\right\}$}

\vskip 1,5mm
\noindent
Call $H_K(f_1,\ldots,f_m;g)$ a {\it Hilbert subset of} $K^r$. If in addition $n=1$ and each $f_i$ is separable in $x$ (i.e., $f_i$ has no multiple root in $\overline{K(\underline \lambda)}$), call 
$H_K(f_1,\ldots,f_m;g)$ a {\it separable Hilbert subset of $K^r$}. The field $K$ is called {\it Hilbertian} if every separable Hilbert subset 
of $K^r$ is nonempty  and {\it totally Hilbertian} if every Hilbert subset of $K^r$ is nonempty ($r\geq 1$).
Equivalently, ``nonempty'' can be replaced by ``Zariski-dense in $K^r$'' in 
the definitions.
As recalled earlier, a field $K$ is totally Hilbertian if and only if it is Hilbertian and the imperfectness condition holds: $K$ is imperfect if
of characteristic $p>0$.

Classical Hilbertian fields include the field $\Qq$, the rational function fields $\Ff_q(u)$ (with $u$ some indeterminate) and all of their finitely generated extensions \cite[Theorem 13.4.2]{FrJa}, every abelian extension of $\Qq$ \cite[Theorem 16.11.3]{FrJa}, fields $k((u_1,\ldots,u_r))$ of formal power series in $r\geq 2$ variables over a field $k$ \cite[Theorem 15.4.6]{FrJa}; all of them are also
totally Hilbertian. Algebraically closed fields, the fields $\Rr$, $\Qq_p$ of real, of $p$-adic numbers, more generally Henselian fields are non-Hilbertian.  The fraction field of a UFD $R$ need not be Hilbertian (take $R=\Zz_p$), even if $R$ has infinitely many distinct prime ideals: a counter-example 
is given in \cite[Example 15.5.8]{FrJa}.

{\it Fields with the product formula} provide other examples of Hilbertian fields. Recall from \cite[\S 15.3]{FrJa} that a nonempty set $S$ of primes ${\frak p}$ of $K$, with associated absolute value $|\hskip 2mm |_{\frak p}$, is said to satisfy the product formula if for each ${\frak p}\in S$, there exists $\beta_{{\frak p}} >0$ such that:
\vskip 1,5mm

\noindent
(1) {\it For each $a\in K^\times$, the set $\{{\frak p}\in S \hskip 2pt | \hskip 2pt |a|_{\frak p} \not=1 \}$ is finite and $\prod_{{\frak p}\in S} |a|_{\frak p}^{\beta_{{\frak p}}} = 1$.}

\vskip 1,5mm

\noindent
In this case call $K$ a field with the product formula. From a result of Weissauer, such fields  are Hilbertian \cite[Theorem 15.3.3]{FrJa}.
The fields $\Qq$, $k(\lambda_1,\ldots, \lambda_r)$ with $k$ any field and $r\geq 1$, and their finite extensions, are fields 
with the product formula.

\subsection{Hilbertian ring} \label{ssec:Hilbertian-UFD} The following definition is given in \cite[\S 13.4]{FrJa}. 
\begin{definition} \label{def:HilbertUFD}
An integral domain $R$ with fraction field $K$ is said to be a {\it Hilbertian ring} if every separable Hilbert subset of $K^r$  ($r\geq 1$) contains $r$-tuples $\underline \lambda^\ast= (\lambda_1^\ast, \ldots, \lambda_r^\ast)$ with coordinates in $R$. 
\end{definition}

Since Zariski open subsets of Hilbert subsets remain Hilbert subsets, it is equivalent to require that a Zariski dense subset of tuples $\underline \lambda^\ast$ exist in Definition \ref{def:HilbertUFD}. 
Under the imperfectness assumption, a better property holds for Hilbertian rings, and extends to arbitrary Hilberts sets. 

\begin{proposition} \label{prop:I-HilbertUFD}
Let $R$ be an integral domain such that the fraction field $K$ is imperfect if of characteristic $p>0$. 
The following are equivalent.

\noindent
{\rm (i)} $R$ is a Hilbertian ring.

\noindent
{\rm (ii)} Every separable Hilbert subset of $K$ contains elements $\lambda^\ast \in R$.

\noindent
{\rm (iii)} For every nonzero $\lambda_0^\ast\in R$ and every $\underline a=(a_1,\ldots,a_r)\in R^r$, every Hilbert subset of $K^r$  {\rm (}$r\geq 1${\rm )} contains $r$-tuples $\underline \lambda^\ast= (\lambda_1^\ast, \ldots, \lambda_r^\ast)$ with nonzero coordinates in $R$ and such that 
$\lambda_i^\ast \equiv a_i\hskip 1mm [{\rm mod}\hskip 1mm \lambda_0^\ast \cdots \lambda_{i-1}^\ast]$, $i=1,\ldots,r$.
\end{proposition}

Clearly, it suffices to prove ${\rm (ii)} \Rightarrow {\rm (iii)}$. This is done in \S \ref{ssec:prop-Hilbertian-UFD}  by reducing the number of variables to reach the separable situation $r=n=1$ of Definition \ref{def:HilbertUFD}.
We recall a classical tool.

\subsection{The Kronecker substitution} \label{ssec:kronecker}
Given an arbitrary field $K$, an irreducible polynomial $f\in K[\underline \lambda,\underline y]$, 
of degree $\geq 1$ in $\underline y = (y_1,\ldots,y_m)$ and an integer $\displaystyle D>\max_{{1\leq i\leq m}}\deg_{y_i}(f)$,
the Kronecker substitution is the map
\vskip 1,5mm

\centerline{$S_D: {\mathcal Pol}_{K(\underline \lambda), m,\underline D} \rightarrow {\mathcal Pol}_{K(\underline \lambda), 1,D^{m}}$, \hskip 2mm with $\underline D=(D,\ldots,D)$,}

\vskip 1,5mm

\noindent
deriving from the substitution of $y^{D^{i-1}}$ for $y_i$, $i=1,\ldots,m$, and leaving the coefficients in the field $K(\underline \lambda)$ unchanged. 

\begin{proposition} \label{prop:kronecker} There exist a finite set ${\mathcal S}(f)$ of irreducible polynomials $g\in K[\underline \lambda][y]$ of degree $\geq 1$ in $y$ and a nonzero polynomial $\varphi \in K[\underline \lambda]$ such that the Hilbert subset ${H}_{K}(f)\subset K^r$ contains the Hilbert subset 

\vskip 1,5mm

\centerline{$H_{K}({\mathcal S}(f);\varphi)$}
\vskip 1,5mm

\noindent
Furthermore, the finite set ${\mathcal S}(f)$ can be taken to be the set of irreducible divisors of $S_D(f)$ in $K[\underline \lambda][y]$.
\end{proposition}

\begin{proof}
See \cite[Lemma 12.1.3]{FrJa}. The statement is only stated for $r=1$ but the proof carries over  to the situation $r\geq 1$ by merely changing the single variable for an $r$-tuple $\underline \lambda=(\lambda_1,\ldots,\lambda_r)$ of variables.
\end{proof}

We will also use several times the following observation.

\begin{lemma} \label{lem:imperfect} Let $R$ be a Hilbertian ring with a fraction field 
$K$ of characteristic $p>0$ and imperfect. There are
infinitely many $a\in R$ that are different modulo $K^p$.
\end{lemma}

\begin{proof}
Let $R$ be a Hilbertian ring. Clearly 
$K$ is Hilbertian, in particular it 
is infinite. Assume further that $K$ is of characteristic $p>0$ and imperfect. Then $K\not= K^p$ and $K/K^p$ is a nonzero vector space over the infinite field $K^p$. 
Thus $K/K^p$ is infinite. It follows 
that if $h\in \Nn$ is an integer, one can find $h+1$ elements $k_1,\ldots,k_{h+1}$ of $K$ that are different
modulo $K^p$. If $\delta \in R$ is a common denominator of $k_1,\ldots,k_{h+1}$, then $\delta k_1,\ldots,\delta k_{h+1}$ are elements
of $R$ that are distinct modulo $K^p$. The conclusion follows.
\end{proof}

\subsection{Proof of Proposition \ref{prop:I-HilbertUFD}} \label{ssec:prop-Hilbertian-UFD} 

Fix an integral domain $R$ satisfying the imperfectness assumption and assume that condition (ii) holds.
Let $\lambda_0^\ast \in R\setminus\{0\}$, $\underline a=(a_1,\ldots,a_r)\in R^r$ and ${\mathcal H}\subset K^r$ be a Hilbert subset.

\subsubsection{First reductions} \label{ssec:first-reductions}
Consider the Hilbert subset ${\mathcal H}_{\lambda_0^\ast, a_1}$ deduced from ${\mathcal H}$ by substituting $\lambda_0^\ast \lambda_1+ a_1$ to $\lambda_1$ in the polynomials involved in ${\mathcal H}$. This first reduction is used at the end of the proof in \S \ref{ssec:end-prop-Hilbertian-UFD}. 

From the standard reduction Lemma 12.1.1  from \cite{FrJa}, the Hilbert
subset ${\mathcal H}_{\lambda_0^\ast, a_1}$ contains a Hilbert subset  of the form
\vskip 1mm

\centerline{$\displaystyle H_{K}(f_1,\ldots,f_m;g)=\left\{ \underline \lambda^\ast \in K^r \hskip 1mm \left | 
\begin{matrix}
f_i(\underline \lambda^\ast, \underline x)\hskip 1mm  \hbox{\rm irreducible in } K[\underline x] \cr
\hskip 1mm  \hbox{for each}\ i=1,\ldots,m, \hfill \cr
g(\underline \lambda^\ast) \not=0 \cr
\end{matrix} \right.
\right\}$}

\vskip 0,5mm

\hskip 7mm with $f_1,\ldots,f_m$ irreducible polynomials in $K[\underline \lambda,\underline x]$, 

\hskip 7mm of degree at least $1$ in $\underline x$ and $g\in K[\underline \lambda]$, $g\not=0$.
\vskip 1,5mm

For $i=1,\ldots,m$, view $f_i$ as a polynomial in $\underline y=(\lambda_2,\ldots,\lambda_r,x_1,\ldots,x_n)$ with coefficients in $K[\lambda_1]$. From Proposition \ref{prop:kronecker},  there is a finite set ${\mathcal S}(f_i)$ of irreducible polynomials $g\in K[\lambda_1][y]$ of degree $\geq 1$ in $y$ and a nonzero polynomial $\varphi _i\in K[\lambda_1]$ such that the Hilbert subset $H_{K}(f_i)\subset K$ contains the Hilbert subset $H_{K}({\mathcal S}(f_i); \varphi_i)\subset K$.

Consider the Hilbert subset
\vskip 1mm

\centerline{${H}_K({\mathcal S}(f_1)\cup\cdots \cup {\mathcal S}(f_m); \varphi_1 \cdots \varphi_m) \subset K$.} 
\vskip 1mm

\noindent
From the standard 
reduction Lemma 12.1.4 from \cite{FrJa}, this Hilbert
subset  contains a Hilbert subset  of the form
\vskip 1mm

\centerline{$\displaystyle H_{K}(g_1,\ldots,g_\nu)=\left\{\lambda_1^\ast \in K \hskip 1mm \left | 
\begin{matrix}
g_i(\lambda_1^\ast, y)\hskip 1mm  \hbox{\rm irreducible in } K[y] \cr
\hskip 1mm  \hbox{\rm for each} \ i=1,\ldots,\nu, \hfill\cr 
\end{matrix} \right.
\right\}$}

\vskip 0,5mm

\hskip 7mm with $g_1,\ldots,g_\nu$ irreducible polynomials in $K[\lambda_1,y]$, 

\hskip 7mm monic and of degree at least $2$ in $y$.
\vskip 1mm

\subsubsection{{\bf 1st case:} $g_1,\ldots,g_\nu$ are separable in $y$} From
 assumption (ii), 
there is an element $\lambda_1^\ast \in R\setminus \{-a_1/\lambda_0^\ast\}$
such that, for each $i=1,\ldots,\nu$,  
$g_i(\lambda_1^\ast,y)$ is irreducible in $K[y]$ and $\deg_{\underline x}(f_i(\lambda_1^\ast, \lambda_2,\ldots,\lambda_r,\underline x)) \geq 1$.
We refer to \S \ref{ssec:end-prop-Hilbertian-UFD} for the end of the proof which is common to 1st and 2nd cases.

\subsubsection{ {\bf 2nd case:} \label{ssec:2nd-case}
 $g_1,\ldots,g_\nu$ are not all separable in $y$.} Necessarily $K$ is of characteristic $p>0$.
The following lemma (which we will use a second time) adjusts arguments from \cite[Prop. 12.4.3]{FrJa}. 
For simplicity, set $\lambda = \lambda_1$.

\begin{lemma} \label{lem:ikeda_argument}
Under the 2nd case assumption, for every nonzero $\lambda_0^\ast \in R$, there is a nonzero $b\in \lambda_0^\ast R$ with this property:
there exist irreducible polynomials $\widetilde Q_1,\ldots,\widetilde Q_\nu$ in $K[\lambda,y]$, separable, monic of degree  $\geq 1$ in $y$ such that for all but finitely many $\tau \in H_K(\widetilde Q_1,\ldots,\widetilde Q_\nu)$,  $\tau^p+b$ is in $H_{K}(g_1,\ldots,g_\nu)$.
\end{lemma}

\begin{proof}[Proof of Lemma \ref{lem:ikeda_argument}]
Assume $g_1,\ldots,g_\ell$ are not separable in $y$ (with $\ell\geq 1$) and $g_{\ell+1},\ldots,g_\nu$ are separable in $y$. 
For each $i=1,\ldots,\ell$, there exists $Q_i\in K[\lambda,y]$ irreducible, separable, monic and of degree 
$\geq 1$ in $y$ and $q_i$ a power of $p$ different from $1$ such that $g_i(\lambda,y)=Q_i(\lambda,y^{q_i})$. Since $g_i(\lambda,y)$ is irreducible in $K[\lambda,y]$, $Q_i$ has a coefficient $h_i\in K[\lambda]$ which is not a $p$th power. Choose $a_i\in R$ 
with $h_i(\lambda+a_i) \in K^p[\lambda]$ if there exists any, otherwise let $a_i=0$. Also set $Q_i=g_i$ for $i=\ell+1,\ldots,\nu$. 

Consider the elements $a\in R$ from Lemma \ref{lem:imperfect}. Among the corresponding elements 
$a\lambda_0^\ast\in R$, which are also different modulo $K^p$, there is at least one, say $b=a\lambda_0^\ast$, 
such that
$b\in R \setminus \bigcup_{i=1}^\ell (a_i+ K^p)$.
By \cite[Lemma 12.4.2(b)]{FrJa}, $h_i(\lambda+b)\notin K^p[\lambda]$,
$i=1,\ldots,\ell$. 

Consider the polynomials $\widetilde Q_i(\lambda,y)=Q_i(\lambda^p+b,y)$, $i=1,\ldots,\nu$. They are monic and separable in $y$. Furthermore, as detailed in \S 12.4 from \cite{FrJa} (and \cite{FrJanew} which clarifies the argument), they are irreducible in $K[\lambda,y]$. 

Let $\tau \in H_K(\widetilde Q_1,\ldots,\widetilde Q_\nu)$ but not in the set $C$, finite by \cite[Lemma 12.4.2(c)]{FrJa}, of all elements $c\in R$ with $h_i(c^p+b)\in K^p$ for some $i=1,\ldots, \ell$. For $i=\ell+1,\ldots,\nu$, we have $\widetilde Q_i(\tau,y)=g_i(\tau^p+b, y)$ 
and so $g_i(\tau^p+b, y)$ is irreducible in $K[y]$. Let $i\in \{1,\ldots, \ell\}$. Since $\tau \notin C$, we have $h_i(\tau^p+b)\notin K^p$. Hence $Q_i(\tau^p+b, y)=\widetilde Q_i(\tau, y) \notin K^p[y]$. From the choice of $\tau$, this polynomial is irreducible in $K[y]$.
By \cite[Lemma 12.4.1]{FrJa},  we obtain that
\vskip 1mm

\centerline{$\widetilde Q_i(\tau,y^{q_i})= Q_i(\tau^p+b,y^{q_i}) = g_i(\tau^p+b,y)$}
\vskip 1mm

\noindent
is irreducible in $K[y]$. Whence finally: $\tau^p+b\in H_K(g_1,\ldots,g_\nu)$. \end{proof}

Use then the assumption (ii) of Proposition \ref{prop:I-HilbertUFD} 
to conclude that
for the element $b$ and the polynomials $\widetilde Q_1,\ldots,\widetilde Q_\nu$ given by Lemma \ref{lem:ikeda_argument},
the Hilbert subset $H_K(\widetilde Q_1,\ldots,\widetilde Q_\nu)$ contains infinitely many elements $\tau \in R$.
Fix one off the finite list of exceptions in the final sentence of Lemma \ref{lem:ikeda_argument} and such that  $\lambda_1^\ast = \tau^p+b$ is different from $-a_1/\lambda_0^\ast$. The element $\lambda_1^\ast\in R$ is then in $H_K(g_1,\ldots,g_\nu)$ and $\lambda_0^\ast \lambda_1^\ast + a_1\not=0$. Up to excluding finitely many more $\tau$ above, we may also assure that $\deg_{\underline x}(f_i(\lambda_1^\ast, \lambda_2,\ldots,\lambda_r,\underline x)) \geq 1$ ($i=1,\ldots,\nu$).
(We have only used here that $b\in R$. The possible choice of $b$ in $\lambda_0^\ast R$ will be used later (\S \ref{ssec:proof1})).

\subsubsection{End of proof of Proposition \ref{prop:I-HilbertUFD}} \label{ssec:end-prop-Hilbertian-UFD} Applying Prop.\ref{prop:kronecker} and taking into account the first reduction changing ${\mathcal H}$ to ${\mathcal H}_{\lambda_0^\ast, a_1}$ yields in both cases that
\vskip 1,5mm

\noindent
(2) there is $\lambda_1^\ast\in R\setminus\{0\}$ such that $\lambda_1^\ast \equiv a_1\hskip 1mm [{\rm mod}\hskip 1mm  \lambda_0^\ast]$, $f_i(\lambda_1^\ast,\lambda_2,\ldots,\lambda_r, \underline x)$ is irreducible in $K[\lambda_2,\ldots,\lambda_r,\underline x]$ and is of degree at least $1$ in $\underline x$, $i=1,\ldots,m$.
\vskip 1,5mm

\noindent
Repeating this argument provides a $r$-tuple $\underline \lambda^\ast= (\lambda_1^\ast,\ldots,\lambda_r^\ast)$ in $(R\setminus \{0\})^r$ such that $f_1(\underline \lambda^\ast,\underline x), \ldots, f_m(\underline \lambda^\ast,\underline x)$ are irreducible in $K[\underline x]$ (so $\underline \lambda^\ast$ is in the original Hilbert subset ${\mathcal H}$) and such that $\lambda_i^\ast \equiv a_i\hskip 1mm [{\rm mod}\hskip 1mm  \lambda_0^\ast\cdots \lambda_{i-1}^\ast]$, $i=1,\ldots,r$.

\subsection{UFD with fraction field with the product formula} \label{sssec:Hilbert} 

\begin{theorem} \label{lemma:hilbert1} If $R$ is an integral domain such that the fraction field $K$ has the product formula and is imperfect if of characteristic $p>0$, then $R$ is a Hilbertian ring.
\end{theorem}

Fix a ring $R$ as in the statement.
Theorem \ref{lemma:hilbert1} relies on the following lemma, whose main ingredient is a result  for fields with the product formula. Recall a useful tool in a field $K$ with a set $S$ of primes $\frak p$ satisfying  the product formula. For every $a\in K$, the ({\it logarithmic}) {\it height} $h(a)$ of $a$ is defined by:
\vskip 1mm

\centerline{$h(a) = \sum_{{\frak p} \in S} \log(\max(1,|a|_{\frak p}))$.}

\vskip 1mm
\noindent
Clearly $h(a^n) = n h(a)$ ($n\in \Nn$) and 
$h(1/a) = h(a)$ if $a\not=0$.

\begin{lemma} \label{thm:indian}
Let $f_1,\ldots,f_m$ be $m$ irreducible polynomials in $K(\lambda)[y]$. For all but finitely many $t_0\in R$, there is a nonzero element $a\in R$ with 
the following property: if $b\in R$ is of height $h(b)>0$, the
Hilbert subset $H_{K}(f_1,\ldots,f_{m})$ contains infinitely many elements of $R$ of 
the form $t_0+ab^\ell$ {\rm (}$\ell>0${\rm )}.
\end{lemma}

\begin{proof}
\cite[Theorem 3.3]{DeIndian} proves the weaker 
version for which the element $a$ is only 
asserted to lie in $K$. 
However the proof can be adjusted so that $a\in R$. 
Specifically, the same argument there
leads to the stronger conclusion provided that, if $K$ is of characteristic $p>0$, infinitely 
many $a\in R$ can be found that are different 
modulo $K^p$.  This is the conclusion of Lemma \ref{lem:imperfect}.
\end{proof}

\begin{proof}[Proof of Theorem \ref{lemma:hilbert1}]
We prove condition (ii) from Proposition \ref{prop:I-HilbertUFD}. Let ${\mathcal H}\subset K$ be a separable Hilbert subset.
From Lemmas 12.1.1 and 12.1.4 of \cite{FrJa}, the Hilbert
subset ${\mathcal H}$ contains a separable Hilbert subset  of the form
\vskip 1mm

\centerline{$\displaystyle H_{K}(f_1,\ldots,f_m)=\left\{ \lambda^\ast \in K \hskip 1mm \left | 
\begin{matrix}
f_i(\lambda^\ast, y)\hskip 1mm  \hbox{\rm irreducible in } K[y] \cr
\hskip 1mm  \hbox{\rm for each} \ i=1,\ldots,m, \hfill\cr
\end{matrix} \right.
\right\}$}

\vskip 0,5mm

\hskip 7mm with $f_1,\ldots,f_m$ irreducible polynomials in $K[\lambda,y]$, 

\hskip 7mm monic, separable and of degree at least $2$ in $y$.
\vskip 1,5mm

Pick an element $t_0\in R$
that avoids the finite set of exceptions in Lemma \ref{thm:indian}.
Consider an element $a\in R$ associated to this $t_0$ in Lemma \ref{thm:indian}. Choose an element 
$b\in R$ of height $h(b)>0$. 

Here is an argument showing that such $b$ exist. 
Fix a prime ${\frak p}\in S$. Recall that by definition, the corresponding absolute value is nontrivial  \cite[\S 13.3]{FrJa}: there exists $b \in K$ such that $|b|_{\frak p} \not=1$. One may assume that $b \in R$. From the product formula, there is a prime ${\frak p}_0\in S$ such that $|b|_{\frak p_0} > 1$. 
We have $h(b)\geq  \log (\max(1,|b|_{\frak p_0})) >0$.

From Lemma \ref{thm:indian}, $\lambda_1^\ast=t_0+ab^\ell\in R$ is in the Hilbert subset $H_{K}(f_1,\ldots,f_{m})$, hence in 
the Hilbert subset ${\mathcal H}$,
for infinitely many 
integers $\ell>0$. 
\end{proof}

\subsection{Polynomial rings in one variable} \label{sssec:Hilbert2}

\begin{theorem}  \label{lemma:hilbert2}
Assume that $R=k[u]$ with $k$ an arbitrary field. Let ${\mathcal H}$ be a Hilbert subset of  $K^r$ {\rm (}$r\geq 1${\rm )},
$\lambda_0^\ast \in R$ a nonzero element of $R$ and $d_1\geq 1$ an integer. Define $\widetilde p$ by
\vskip 1mm

\centerline{$\widetilde p = \left\{\begin{matrix}
1\hskip 5mm  \hbox{if ${\rm char}(k) = 0$ or ${\mathcal H}$ is a separable Hilbert subset} \hfill \\
p\hskip 5mm  \hbox{otherwise}. \hfill \\
\end{matrix}
\right.$}
\vskip 1mm

\noindent
Denote the subset of ${\mathcal H}$ of 
$r$-tuples $\underline \lambda^\ast= (\lambda_1^\ast,\ldots,\lambda_r^\ast)\in R^r$ such that 
$\lambda_1^\ast$ and $\lambda_0^\ast \lambda_2^\ast \cdots \lambda_r^\ast$ are relatively prime in $R$
and  $\max_{1\leq i\leq r} \deg(\lambda_i^\ast) =  \widetilde  p d_1$ by
${\mathcal H}_{\lambda_0^\ast, \widetilde p d_1}$. 
There is an integer $d_0$ such that if $d_1\geq d_0$, the set ${\mathcal H}_{\lambda_0^\ast, \widetilde p d_1}$ 
is nonempty.
\end{theorem}

When $R=k[u]$, statement (iii) from Proposition \ref{prop:I-HilbertUFD} also holds for the Hilbert subset ${\mathcal H}$: there the congruence conditions are stronger but no control is given on the degree in $u$ of $\lambda_1^\ast,\ldots,\lambda_r^\ast$ as in Theorem \ref{lemma:hilbert2}.

We divide the proof of Theorem \ref{lemma:hilbert2} into two parts. The situation: one parameter, one variable, is considered in \S \ref{ssec:proof1}, 
the general one  in \S \ref{ssec:proof2}.

\subsubsection{Proof of Theorem \ref{lemma:hilbert2} {\rm -- {situation} $r=n=1$ --}} \label{ssec:proof1}
We are given a Hilbert subset ${\mathcal H}\subset K=k(u)$, a nonzero element $\lambda_0^\ast \in k[u]$, an integer $d_1\geq 1$ and we need to find an element $\lambda_1^\ast \in k[u]$ such that $\lambda_1^\ast \in {\mathcal H}$, $\lambda_1^\ast$ and $\lambda_0^\ast$ are relatively prime and $\deg(\lambda_1^\ast)= \widetilde p d_1$. 

From Lemmas 12.1.1 and 12.1.4 from \cite{FrJa}, the Hilbert
subset ${\mathcal H}$ contains a Hilbert subset  of the form
\vskip 1mm

\centerline{$\displaystyle H_{K}(f_1,\ldots,f_m)=\left\{ \lambda^\ast \in K \hskip 1mm \left | 
\begin{matrix}
f_i(\lambda^\ast, y)\hskip 1mm  \hbox{\rm irreducible in } K[y] \cr
\hskip 1mm  \hbox{\rm for each} \ i=1,\ldots,m. \hfill \cr
\end{matrix} \right.
\right\}$}

\vskip 0,5mm

\hskip 7mm with $f_1,\ldots,f_m$ irreducible polynomials in $K[\lambda,y]$, 

\hskip 7mm monic and of degree at least $2$ in $y$.
\vskip 1,5mm

We distinguish the two cases corresponding to the definition of $\widetilde p$.
\vskip 1,5mm

\noindent
{\textbf{Separable case}: ${\rm char}(k) = 0$ or ${\mathcal H}$ is a separable Hilbert subset}. As $n=1$, the Hilbert subset ${\mathcal H}$ is also separable under the assumption ${\rm char}(k) = 0$. So we may assume that the polynomials $f_1,\ldots,f_m$ above are separable in $y$. We distinguish two sub-cases.
\vskip 1,5mm

- {\it 1st sub-case}: $k$ is infinite. Use \cite[Prop.4.1 p.236]{langFG}
to assert that there exists a nonempty Zariski open subset  $V\subset \Aa_{k}^2$ such that for all but finitely many $\gamma \in k$, 
\vskip 1mm

\centerline{$\{\tau+\gamma (u-\beta)^{d_1} \in k[u]\hskip 2pt | \hskip 2pt (\tau,\beta) \in V\}\subset H_{K}(f_1,\ldots,f_{m})$.}
\vskip 1mm

\noindent
Fix a nonzero $\gamma \in k$ off the finite exceptional list. There are infinitely many different $(\tau,\beta) \in V$ such that 
no root in $\overline k$ of the polynomial $\lambda_0^\ast\in k[u]$ is a root of $\tau+\gamma (u-\beta)^{d_1}$, and so 
$\tau+\gamma (u-\beta)^{d_1}$ and $\lambda_0^\ast$ are relatively prime. The corresponding elements $\lambda_1^\ast=
\tau+\gamma (u-\beta)^{d_1}$ are infinitely many different elements of the set ${\mathcal H}_{\lambda_0^\ast, d_1}$.
In this case, one can take $d_0=1$.

\vskip 1,5mm

- {\it 2nd sub-case}: $k$ is finite.
Start with another classical reduction, namely \cite[Lemma 13.1.2]{FrJa}, to conclude that there exist 
polynomials $Q_1, \ldots, Q_\nu$ in $K[\lambda,y]$, irreducible in $\overline K[\lambda,y]$, monic and separable
in $y$, of degree $\geq 2$ in $y$ and such that the Hilbert subset $H_{K}(f_1,\ldots,f_m)$
contains the set 
\vskip 1mm

\centerline{$\displaystyle H^\prime_{K}(Q_1,\ldots,Q_\nu)=\left\{ \lambda^\ast \in K \hskip 1mm \left |  \hskip 1mm
\begin{matrix}
Q_i(\lambda^\ast, y)\hskip 1mm  \hbox{\rm has no root in } K \cr
\hbox{for each}\ i=1,\ldots,\nu \cr
\end{matrix} \right.
\right\}$}

Consider the set $\{{\frak p}_i \hskip 1mm | \hskip 1mm i\in I\}$ of irreducible factors of the given polynomial $\lambda_0^\ast\in k[u]$;
view them as primes of $K$. Apply \cite[Lemma 13.3.4]{FrJa} to assert that, for each $j=1,\ldots,\nu$, there are infinitely primes ${\frak p}_j$ of $K$ such that there is an $a_{{\frak p}_j}\in R$ with this property: if $a \in R$ satisfies $a\equiv a_{{\frak p}_j}\ {\rm mod}\ {\frak p}_j$, then $Q_j(a,v)\not=0$ for every $v\in K$. For each $j=1,\ldots, \nu$, pick one such prime ${\frak p}_j$ that is different from all primes ${\frak p}_i$ with $i\in I$.

Denote the ideal $(\prod_{j=1}^\nu {\frak p}_j)(\prod_{i\in I} {\frak p}_i) \subset R$ by ${\mathcal I}$. From the Chinese Remainder Theorem, there exists $a_0\in R$ such that every $a\in a_0+ {\mathcal I}$ satisfies

\vskip 1mm

\noindent
\centerline{$\displaystyle \left\{\begin{matrix}
a\equiv a_{{\frak p}_j}\ {\rm mod}\ {\frak p}_j\ {\rm for}\ j=1,\ldots, \nu, \hfill\\
a\equiv 1\ {\rm mod}\ {\frak p}_i\ {\rm for}\ i\in I. \hfill\\
\end{matrix}
\right.
$}

\vskip 1mm

\noindent
Consider such an $a$ and rename it $\lambda_1^\ast$. It follows from the first condition that $\lambda_1^\ast\in H^\prime_{K}(Q_1,\ldots,Q_\nu)$ and so $\lambda_1^\ast\in H_{K}(f_1,\ldots,f_m)\subset {\mathcal H}$. It follows from the second condition that $\lambda_1^\ast \not\equiv 0\ {\rm mod}\ {\frak p}_i$ for every $i\in I$. 
Hence $\lambda_1^\ast$ and $\lambda_0^\ast$ are relatively prime. 
Finally when $\lambda_1^\ast=a$ ranges over $a_0+ {\mathcal I}$, $\deg(\lambda_1^\ast)$ assumes all but finitely many 
values in $\Nn$. Therefore there is an integer $d_0$ such that ${\mathcal H}_{\lambda^\ast_0,d_1}\not= \emptyset$ for every $d_1\geq d_0$.
\vskip 1,5mm

\noindent
{\textbf{2nd case}}: {\it ${\rm char}(k) = p >0 $ and ${\mathcal H}$ is not a separable Hilbert subset}. 
Not all the polynomials $f_1,\ldots,f_m$ are separable in $y$.  
Proceed as in 
\S \ref{ssec:2nd-case}. From Lemma 
\ref{lem:ikeda_argument}, there is a nonzero $b\in \lambda_0^\ast R$ and some irreducible polynomials $\widetilde Q_1,\ldots,\widetilde Q_m$ in $K[\lambda,y]$, separable, monic of degree  $\geq 1$ in $y$ such that for all but finitely many $\tau \in H_K(\widetilde Q_1,\ldots,\widetilde Q_m)$,  $\tau^p+b$ is in $H_{K}(f_1,\ldots,f_m)$.

From the separable case of the current proof, there is an integer $d_0\geq 1$ with the following property: the Hilbert subset $H_K(\widetilde Q_1,\ldots,\widetilde Q_\nu)$ contains infinitely many elements $\tau \in R$ such that $\tau$ and $\lambda_0^\ast$ are relatively prime and $\deg(\tau)=d_1$. Fix one off the finite list of exceptions in the final sentence of Lemma \ref{lem:ikeda_argument} and set  $\lambda_1^\ast = \tau^p+b$. We then have $\lambda_1^\ast \in H_K(f_1,\ldots,f_m)$. 
Furthermore $\lambda_1^\ast$ and $\lambda_0^\ast$ are relatively prime in $R$. Finally assuming that $d_0$ is also 
larger than $\deg(b)$, we have $\deg(\lambda_1^\ast)=pd_1$ if $d_1\geq d_0$, thus finally proving that $\lambda_1^\ast \in {\mathcal H}_{\lambda_0^\ast, p d_1}$.

\subsubsection{Proof of Theorem \ref{lemma:hilbert2} {\rm -- situation $r\geq 1$, $n\geq 1$ --}} \label{ssec:proof2}
As in \S \ref{ssec:proof1} we distinguish two cases according to the definition of $\widetilde p$. 

\vskip 1,5mm

\noindent
{\textbf{Separable case}: ${\mathcal H}$ is a separable Hilbert subset {\rm (in particular $n=1$)}.  From  Lemma 12.1.1 and Lemma 12.1.4 from \cite{FrJa}, the separable Hilbert
subset ${\mathcal H}\subset K^r$ contains a Hilbert subset  of the form
\vskip 1mm

\centerline{$\displaystyle H_{K}(f_1,\ldots,f_m)=\left\{ \underline \lambda^\ast \in K^r \hskip 1mm \left | 
\begin{matrix}
f_i(\underline \lambda^\ast, x)\hskip 1mm  \hbox{\rm irreducible in } K[x] \cr
\hskip 1mm  \hbox{\rm for each} \ i=1,\ldots,m, \hfill \cr
\end{matrix} \right.
\right\}$}

\vskip 0,5mm

\hskip 7mm with $f_1,\ldots,f_m$ irreducible polynomials in $K[\underline \lambda,x]$, 

\hskip 7mm separable, monic and of degree at least $2$ in $x$.
\vskip 1,5mm

Set ${\mathcal K}= K(\lambda_3,\ldots,\lambda_{r})$ (with ${\mathcal K}= K$ if $r=2$) and regard $f_1,\ldots,f_m$ as polynomials in the ring ${\mathcal K}(\lambda_{1})[\lambda_2,x]$. By \cite[Proposition 13.2.1]{FrJa}, there exists a nonempty Zariski open subset $U\subset \Aa_{{\mathcal K}}^2$ such that
\vskip 1mm

\centerline{$\{a+b\lambda_{1} \hskip 2pt | \hskip 2pt (a,b) \in U\}\subset H_{{\mathcal K}(\lambda_{1})}(f_1,\ldots,f_m)$.}
\vskip 1mm

\noindent 
Furthermore, up to shrinking $U$, one may require that the polynomials 
\vskip 1mm

\noindent
(4) \hskip 25mm $f_i(\lambda_1,a \lambda_1+b, \lambda_3, \ldots,\lambda_{r},x)$, $i=1,\ldots,m$

\vskip 1mm

\noindent
are separable and of degree at least $2$ in $x$, and that $b\not=0$. As $R=k[u]\subset {\mathcal K}$ is infinite,  
the open subset $U$ contains elements $(a,b)\in R^2$. For such $(a,b)$, the polynomials above in (4)
  are in $K[\lambda_1,\lambda_3,\ldots,\lambda_{r},x]$ and are 
   irreducible in $K(\lambda_1,\lambda_3,\ldots,\lambda_{r})[x]$.
Repeating this procedure provides an  $(r-1)$-tuple $((a_2,b_2),\ldots, (a_r,b_r)) \in (R^2)^{r-1}$ with $b_2\cdots b_r\not=0$ such that the polynomials
\vskip 1,5mm

\centerline{$g_i(\lambda_1,x)=f_i(\lambda_1, a_2\lambda_1+b_2,\ldots, a_r\lambda_1+b_r,x)$, $i=1,\ldots,m$}
\vskip 1,5mm

\noindent
are in $K[\lambda_1,x]$, irreducible in $K(\lambda_1)[x]$, separable and of degree $\geq 2$ in $x$. 

From the proof in situation $r=n=1$ and in the separable case (in \S \ref{ssec:proof1}), there is an integer $\delta_0\geq 1$ with this property:
the Hilbert subset $H_K(g_1,\ldots,g_m)$ contains 
an element $\lambda_1^\ast \in R$ relatively prime to 
$\lambda_0^\ast \cdot b_2 \cdots b_r$ and such that $\deg(\lambda_1^\ast)=\delta_1$ if $\delta_1\geq \delta_0$. Request further to $\delta_0$ to satisfy: 
\vskip 1mm

\noindent
(5) \hskip 30mm $\delta_0 > \max_{2\leq i \leq r} \deg(b_i)$.

\vskip 1mm

Set $d_0 = \delta_0 + \max_{2\leq i \leq r} \deg(a_i)$ and fix an integer $d_1\geq d_0$. It follows from $d_1- \max_{2\leq i \leq r} \deg(a_i)\geq \delta_0$ that the Hilbert subset $H_K(g_1,\ldots,g_m)$ contains 
an element $\lambda_1^\ast \in R$ 
such that $\deg(\lambda_1^\ast)=d_1-\max_{2\leq i \leq r} \deg(a_i)$.

Consequently we have the following:

\noindent
-  the $r$-tuple $\underline \lambda^\ast = (\lambda_1^\ast, a_2\lambda_1^\ast+b_2, \ldots,a_{r-1}\lambda_1^\ast+b_{r-1},a_r \lambda_1^\ast+b_r)\in R^r$ is in the original Hilbert subset ${\mathcal H}$, 
and, 
denoting the $i$-th component of $\underline \lambda^\ast$ by $\lambda_i^\ast$,

\noindent
- $\lambda_1^\ast$ is relatively prime to $\lambda_0^\ast \lambda_2^\ast \cdots \lambda_r^\ast$,

\noindent
- the largest degree of $\lambda_1^\ast, \ldots, \lambda_r^\ast$ is $d_1$ (due to condition (5),
this largest degree is $\max_{2\leq i \leq r} \deg(a_i \lambda_1^\ast)$).

\noindent
This proves that $\underline \lambda^\ast  \in {\mathcal H}_{\lambda_0^\ast, d_1}$.
\vskip 2mm

\noindent
{\textbf{General case}: We will use the Kronecker substitution.  
The Hilbert subset ${\mathcal H}$ contains a Hilbert subset
\vskip 1mm

\centerline{$\displaystyle H_{K}(f_1,\ldots,f_m;g)=\left\{ \underline \lambda^\ast \in K^r \hskip 1mm \left | 
\begin{matrix}
\hskip 1mm f_i(\underline \lambda^\ast, \underline x)\hskip 1mm  \hbox{\rm irreducible in } K[\underline x] \cr
\hskip 1mm \hbox{\rm for each}\ i=1,\ldots,m, \hfill \cr
g(\underline \lambda^\ast) \not=0 \cr
\end{matrix} \right.
\right\}$}

\vskip 0,5mm

\hskip 7mm with $f_1,\ldots,f_m$ irreducible polynomials in $K[\underline \lambda,\underline x]$, 

\hskip 7mm of degree at least $1$ in $\underline x$ and $g\in K[\underline \lambda]$, $g\not=0$.
\vskip 1,5mm

As  in \S \ref{ssec:prop-Hilbertian-UFD}, Proposition \ref{prop:kronecker}, followed by  \cite[Lemma 12.1.4]{FrJa}, provides polynomials $g_1,\ldots,g_\nu$, irreducible in $K[\lambda_1,y]$, 
monic and of degree $\geq 2$ in $y$ with this property. For every $\lambda_1^\ast \in H_K(g_1,\ldots,g_\nu)$,
each of the polynomials
\vskip 1mm

\centerline{$f_i(\lambda_1^\ast,\lambda_2,\ldots,\lambda_r, \underline x)$, $i=1,\ldots,m$,}
\vskip 1mm

\noindent
is irreducible in $K[\lambda_2,\ldots,\lambda_r, \underline x]$. From the proof in situation $r=n=1$ (\S \ref{ssec:proof1}),  
the Hilbert subset $H_K(g_1,\ldots,g_\nu)$ contains infinitely many $\lambda_1^\ast \in R$ relatively prime to 
$\lambda_0^\ast$.
Repeating this argument $(r-2)$ times provides $\lambda_1^\ast,\ldots,\lambda_{r-1}^\ast\in R$ 
such that  $f_i(\lambda_1^\ast,\ldots, \lambda_{r-1}^\ast,\lambda_r, \underline x)$ is irreducible in $K[\lambda_r, \underline x]$ ($i=1,\ldots,m$) and $\lambda_i^\ast$ and $\lambda_0^\ast \lambda_1^\ast \cdots \lambda_{i-1}^\ast$ are relatively prime ($i=1,\ldots,r-1$).

Repeating the argument once more but applying this time the full conclusion of the case $r=n=1$ of the proof including the degree condition, we obtain that there is an integer $d_{0}$, which we may also choose to be larger than $\max_{1\leq i\leq r-1} \deg(\lambda_i^\ast)$, with the following property: if $d_1\geq d_{0}$, there exists an element $\lambda_r^\ast\in R$ such that

\noindent
- $f_i(\lambda_1^\ast,\ldots, \lambda_{r-1}^\ast,\lambda_r^\ast, \underline x)$ is irreducible in $K[\underline x]$, $i=1,\ldots,m$,}

\noindent
- $\lambda_r^\ast$ and $\lambda_0^\ast \lambda_1^\ast \cdots \lambda_{r-1}^\ast$ are relatively prime,

\noindent
- $\deg(\lambda_r^\ast) = \widetilde pd_1$.

\noindent
Finally the $r$-tuple $\underline \lambda^\ast$ is in the original Hilbert subset ${\mathcal H}$, $\lambda_i^\ast$ and $\lambda_0^\ast \lambda_1^\ast \cdots \lambda_{i-1}^\ast$ are relatively prime ($i=1,\ldots,r$), and consequently, $\lambda_1^\ast$ is relatively prime to $\lambda_0^\ast \lambda_2^\ast \cdots \lambda_r^\ast$, and $\max_{1\leq i\leq r} \deg(\lambda_i^\ast) =  \widetilde  p d_1$. Thus
 the set ${\mathcal H}_{\lambda_0^\ast, d_1}$ is nonempty.

\section{Proofs of the main results} \label{sec:Schinzel}

\subsection{Proof of Theorem \ref{thm:schinzel1} and Theorem \ref{thm:schinzel2}} \label{ssec:schinzel_proof}
Recall the notation from \S \ref{sec:preliminary}: $R$ is a UFD with fraction field $K$, $\underline x = (x_1,\ldots,x_n)$, $\lambdabar=(\lambda_0, \lambda_1,\ldots,\lambda_\ell)$ ($n\geq 1$, $\ell \geq 1$) are two tuples of indeterminates, ${\underline Q}= (Q_0, Q_1,\ldots,Q_\ell$), with $Q_0=1$,  is a $(\ell+1)$-tuple of nonzero polynomials in $R[\underline x]$, distinct up to multiplicative 
constants in $K^\times$, $\underline{P}=\{P_1,\ldots, P_s\}$ is a set of $s$ polynomials 
\vskip 1mm

\centerline{$P_i(\underline x,y)= P_{i\degP_i}(\underline x) \hskip 2pt y^{\degP_i} +\cdots+ P_{i1}(\underline x) \hskip 1pt y + P_{i0}(\underline x)$,}
\vskip 1mm

\noindent
irreducible in $R[\underline x,y]$ and of degree $\degP_i \geq 1$ in $y$, $i=1,\ldots,s$. We also set
\vskip 1mm

\centerline{$M(\lambdabar,\xbar) = \sum_{j=0}^{\ell} \lambda_j \hskip 1pt Q_j(\underline x)$}
\vskip 1mm

\noindent
and, for $i=1,\ldots,s$,

\vskip 1mm

\centerline{$F_i(\lambdabar,\xbar) 
= P_i(\underline x, M(\lambdabar,\xbar))= P_i(\underline x, \sum_{j=0}^\ell \lambda_j Q_j(\underline x))$}
\vskip 1mm

\noindent The polynomials $F_1,\ldots,F_s$ are irreducible in $R[\underline \lambda, \underline x]$ (Lemma \ref{lemma:first}). 
Finally, for $\underline F=\{F_1,\ldots,F_s\}$, we introduced the subset 
\vskip 1mm

\centerline{${H}_{R}({\underline F})  \subset R^{{\ell}+1}$} 
\vskip 1mm

\noindent
of all $(\ell+1)$-tuples $\underline \lambda^\ast$ (or equivalently, of polynomials 
$\Lambda (\underline x) = \sum_{j=0}^\ell \lambda_i^\ast \hskip 2pt Q_j(\underline x)$) such that 
$F_i(\underline \lambda^\ast, \underline x) = P_i(\underline x, \Lambda(\underline x))$ is irreducible in $R[\underline x]$, $i=1,\ldots,s$.

Given a nonzero element $\lambda_{-1}^\ast\in R$ and a tuple  $\underline a=(a_0,\ldots,a_\ell)\in R^{\ell+1}$, consider the subset 
\vskip 1mm

\centerline{${H}_{R,\lambda_{-1}^\ast,\underline a}({\underline F}) \subset {H}_{R}({\underline F})$} 
\vskip 1mm

\noindent
of those $(\ell+1)$-tuples $\underline \lambda^\ast=(\lambda_0^\ast,\ldots,\lambda_\ell^\ast) \in {H}_{R}({\underline F})$ which further satisfy
the congruences $\lambda_i^\ast \equiv a_i\hskip 1mm [{\rm mod}\hskip 1mm \lambda_{-1}^\ast \lambda_0^\ast \cdots \lambda_{i-1}^\ast]$, $i=0,\ldots,\ell$.

\vskip 1mm 

Make this additional assumption on $Q_0,\ldots,Q_\ell$ (which implies $\ell \geq 2$):
\vskip 1,5mm

\noindent
(1) \hskip 2mm {\it $Q_0,\ldots,Q_\ell$ are monomials with coefficient $1$, $Q_0=1$ and
\vskip 1mm

\centerline{$\displaystyle \min(\deg(Q_1), \deg(Q_2)) > \max_{1\leq i\leq s} \deg_{\underline x}(P_i)$.}}

\vskip 1,5mm

\begin{theorem} \label{thm:general} Let $\lambda_{-1}$ be a nonzero element of $R$ and  $\underline a = (1,\ldots,1)\in R^{\ell+1}$.
\vskip 0,5mm

\noindent
{\rm (a)} Assume that $R$ is a UFD and a Hilbertian ring and $K$ is imperfect if it is of cha\-rac\-te\-ristic $p>0$. The subset
${H}_{R,\lambda_{-1}^\ast,\underline a}({\underline F})$ is Zariski-dense in $R^{\ell+1}$. 
\vskip 1mm

\noindent
{\rm (b)} If $R=k[u]$ with $k$ an arbitrary field and $d_1$ is a suitably large integer, then
${H}_{R}({\underline F})$  contains
a polynomial $\Lambda = \sum_{j=0}^\ell \lambda_i^\ast \hskip 2pt Q_j(\underline x)$ with $\underline \lambda^\ast=(\lambda_0^\ast, \ldots, \lambda_\ell^\ast)\in R^{\ell+1}$
such that $\lambda_1^\ast$ and $\lambda_{-1}^\ast\lambda_0^\ast \lambda_2^\ast \cdots \lambda_\ell^\ast $ are relatively prime and $\deg_u(\Lambda) = \widetilde  p d_1$.
\end{theorem}

\begin{proof} 
The number of monomials $Q_i$ is $\ell+1\geq 3$. Each $F_i$ is of degree $\geq 1$ in $\underline x$ and is irreducible in $K(\underline \lambda)[\underline x]$, $i=1,\ldots,s$ (Lemma \ref{lemma:first}).  Let $g\in K[\underline \lambda]$ be a nonzero 
polynomial and consider the Hilbert subset 
\vskip 1mm

\centerline{${H}_{K}({\underline F};g)  \subset K^{{\ell}+1}$.} 
\vskip 1mm

In situation (a), it follows from Proposition \ref{prop:I-HilbertUFD} that the Hilbert subset ${H}_{K}({\underline F};g)$ contains
an $(\ell +1)$-tuple $\underline \lambda^\ast= (\lambda_0^\ast,\ldots,\lambda_\ell^\ast)\in R^{\ell+1}$ satisfying the congruences 
$\lambda_i^\ast \equiv 1\hskip 1mm [{\rm mod}\hskip 1mm \lambda_{-1}^\ast \lambda_0^\ast \cdots \lambda_{i-1}^\ast]$,  $i=0,\ldots,\ell$.

In situation (b), from Theorem \ref{lemma:hilbert2}, the Hilbert subset ${H}_{K}({\underline F};g)$ contains an $(\ell +1)$-tuple $\underline \lambda^\ast$ such that $\lambda_1^\ast$ and $\lambda_{-1}^\ast\lambda_0^\ast \lambda_2^\ast \cdots \lambda_\ell^\ast$ are relatively prime and $\displaystyle \max_{0\leq i\leq \ell} \deg(\lambda_i^\ast) =  \widetilde  p d_1$, i.e,
$\deg_{u}(\Lambda)  = \widetilde  p d_1$ for $\Lambda = \sum_{j=0}^\ell \lambda_i^\ast \hskip 2pt Q_j(\underline x)$. 

Each $F_i(\underline \lambda^\ast,\underline x)$ being irreducible in $K[\underline x]$, to finish the proof, it suffices to 
show that $F_i(\underline \lambda^\ast,\underline x)$ is primitive \hbox{w.r.t.} $R$ ($i=1,\ldots, s$).

Assume otherwise, \hbox{i.e.,} for some $i=1,\ldots,s$, there is an irreducible element $\pi\in R$ dividing all the coefficients 
of $F_i(\underline \lambda^\ast,\underline x)$. The quotient ring $\overline R=R/(\pi)$ is an integral domain. Use the notation $\overline U$ to denote the class modulo $(\pi)$ of polynomials $U$ with coefficients in $R$.
We have: \vskip 1mm

\noindent
(2) \hskip 12mm $\overline P_{i\degP_i}(\underline x) \hskip 2pt \overline M(\lambdabar^\ast,\xbar)^{\degP_i} +\cdots+ \overline P_{i1}(\underline x) \hskip 1pt \overline M(\lambdabar^\ast,\xbar) = - \overline P_{i0}(\underline x)$.
\vskip 1mm

\noindent
We distinguish two cases.

\vskip 1,5mm
 
 \noindent
  {\it 1st case}: $\pi$ divides all polynomials $P_{ij}(\underline x)$, $j=1,\ldots,\degP_i$. From (2), $\pi$ also
divides $P_{i0}(\underline x)$. This contradicts $P_i(\underline x,y)$ being primitive \hbox{w.r.t.} $R$.
\vskip 1,5mm
 
 \noindent
{\it 2nd case}: there is an index $j\in \{1,\ldots, \degP_i\}$ such that $\pi$ does not divides $P_{ij}(\underline x)$.
As $\lambda_1^\ast$ and $\lambda_2^\ast$ are relatively prime (in both situations (a) and (b)), one of the two is not divisible by $\pi$. Conjoin this with our monomials $Q_i$ being of coefficient $1$ to conclude that $\overline M(\lambdabar^\ast,\xbar)\not=0$ in $R/(\pi)[\underline x]$ and that there is at least one nonzero term $\overline P_{ij}(\underline x) \hskip 2pt \overline M(\lambdabar^\ast,\xbar)^{j}$ with $j\in \{1,\ldots, \degP_i\}$. 
Furthermore we have:
\vskip 1mm

\centerline{$\deg (\overline M(\underline \lambda^\ast,\underline x)) \geq \min (\deg(Q_1), \deg(Q_2))$.}
\vskip 1mm

\noindent
Using next the following inequality (coming from assumption (1)):

\vskip 1,5mm

\centerline{$\displaystyle \min(\deg(Q_1), \deg(Q_2)) > \max_{\genfrac{}{}{0pt}{}{1\leq i\leq s}{1\leq j\leq \degP_i}} \deg(P_{ij})$,}

\vskip 1,5mm

\noindent
we obtain that all nonzero terms $\overline P_{ij}(\underline x) \hskip 2pt \overline M(\lambdabar^\ast,\xbar)^{j}$ with $j\in \{1,\ldots, \degP_i\}$ are of different degrees: otherwise,  for two integers $j,k \in \{1,\ldots, \degP_i\}$ with $k>j$, we would have the following, where $\delta= \deg(\overline M(\lambdabar^\ast,\xbar))$:
\vskip 1,5mm

\centerline{$\displaystyle \max_{\genfrac{}{}{0pt}{}{1\leq i\leq s}{1\leq h\leq \degP_i}} \deg({P}_{ih}) \geq \deg(\overline{P}_{ij})-\deg(\overline{P}_{ik}) = (k-j)\delta \geq \delta$,}
\vskip 1mm

\noindent
which contradicts the preceding inequalities. It follows that the left-hand side of (2) is of degree $\geq \deg (\overline M(\lambdabar^\ast,\xbar))$. 
But then the following inequality (using again assumption (1)):

\vskip 1mm

\centerline{$\displaystyle \deg (\overline M(\underline \lambda^\ast,\underline x)) \geq \min (\deg(Q_1), \deg(Q_2)) >  \max_{1\leq i\leq s} \deg(P_{i0}(\underline x))$.}
\vskip 1mm

\noindent
contradicts identity (2).
\end{proof}

\begin{proof}[Proof of Theorem \ref{thm:schinzel1} and Theorem \ref{thm:schinzel2}] 
From Theorem \ref{lemma:hilbert1}, the assumption on $R$ in Theorem \ref{thm:schinzel1} implies that of Theorem \ref{thm:general}(a), and $R=k[u]$ in both Theorem \ref{thm:schinzel2} and Theorem \ref{thm:general}(b). Theorems \ref{thm:schinzel1} and \ref{thm:schinzel2} then correspond to the special case of Theorem \ref{thm:general} for which, for a given $\underline d\in (\Nn^\ast)^n$,  the $Q_i$ are all the monomials
 $Q_0,Q_1,\ldots,Q_{N_{\underline d}}$ in ${\mathcal Pol}_{k,n,\underline d}$
and $Q_1$, $Q_2$ are monomials of  degree 
$d_1+\cdots+d_n$ and $d_1+\cdots+d_n-1$. Assumption on $d_1,\ldots,d_n$ in Theorems \ref{thm:schinzel1} and \ref{thm:schinzel2} guarantees assumption (1) of Theorem \ref{thm:general}.
\end{proof}

\begin{remark} 
\noindent
The proof shows that Theorem \ref{thm:schinzel1} holds under the more general assumption that $R$ is a UFD, a Hilbertian ring and $K$ is imperfect if of characteristic $p>0$. We note that there exist UFD with a Hilbertian 
fraction field satisfying the imperfectness assumption but not Hilbertian as a ring, e.g. the ring $\Cc[[u_1,\ldots,u_n]]$ of formal power series with $n\geq 2$ \cite[Example 15.5]{FrJa}. It is unclear whether Theorem \ref{thm:schinzel1} holds for these 
rings.
\end{remark}

\subsection{The multivariable Schinzel hypothesis} \label{ssec:multi_schinzel}
Theorem \ref{thm:general} offers more flexibility than Theorem \ref{thm:schinzel1} and Theorem \ref{thm:schinzel2}. Instead
of taking for $Q_0,\ldots,Q_\ell$ all the monomials in ${\mathcal Pol}_{k,n,\underline d}$, one may want to work with a proper subset of them and construct irreducible polynomials of the form $P_i(\underline x, M(\underline x))$ with some of the coefficients  in $M(\underline x)$ equal to $0$.

In this manner one can  extend Theorems \ref{thm:schinzel1} and \ref{thm:schinzel2} to the situation that $P_1,\ldots,P_s$ are polynomials in $m$ variables $y_1,\ldots,y_m$. 

Let $R$ be a UFD with fraction field a field $K$ with the product formula, imperfect if $K$ is of characteristic $p>0$. Let $\underline x=(x_1,\ldots,x_n)$ {\rm (}$n\geq 1${\rm )} and $\underline y=(y_1,\ldots,y_m)$  {\rm (}$m\geq 1${\rm )} be two tuples of 
indeterminates. 

\begin{theorem} \label{thm:schinzel3}
Let $\underline P=\{P_1,\ldots,P_s\}$ be a set of polynomials, irreducible in $R[\underline x,\underline y]$ 
and of degree $\geq 1$ in $\underline y$. 
Let ${\mathcal Irr}_{n}(R,\underline P)$ be the set of all $m$-tuples $\underline M = (M_1,\ldots,M_m)\in R[\underline x]^m$ such that $P_i(\underline x, \underline M(\underline x))$ is irreducible in $R[\underline x]$, $i=1,\ldots, s$.
For every $\underline d\in (\Nn^\ast)^n$ such that 
\vskip 1mm

\centerline{$D := \displaystyle d_1+\cdots+d_n \geq \max_{1\leq i \leq s} (\deg (P_i)+2)$,} 
\vskip 1mm

\noindent
the set 
$ {\mathcal Irr}_{n}(R,\underline P)$ is 
Zariski-dense in ${\mathcal Pol}_{R,n,\underline d} \times \cdots \times {\mathcal Pol}_{R,n,D^{m-1} \underline d}$.
 \end{theorem}
 
 The proof is an easy induction left to the reader: use Theorem \ref{thm:general} to successively specialize  in $R[\underline x]$ the indeterminates $y_1,\ldots,y_m$. 
 
 \subsection{Proof of Corollary \ref{cor:goldbach} (Goldbach)}  \label{ssec:proof-goldbach}
 
Fix an integral domain $R$ as in Theorem \ref{thm:schinzel1}, an integer $n\geq 1$ and a nonconstant polynomial ${\mathcal Q}\in R[\underline x]$. 

 Let $\underline P = \{P_1,P_2\}$ with $P_1= -y$ and $P_2=y+{\mathcal Q}$. We will proceed as in Theorem \ref{thm:general} but with only two monomials $Q_0,Q_1$  (so $\ell = 1$) and without assuming condition (1) from \S \ref{ssec:schinzel_proof}. 
 
 Assume that we are not in the case $n=1=\deg({\mathcal Q})$; this case is dealt with separately. Let $Q_\infty$ be a monic nonconstant monomial appearing in ${\mathcal Q}$ with a nonzero coefficient. Denote this coefficient by $q_\infty$. Let $Q_1$ be a nonconstant monomial distinct from $Q_\infty$  and of degree $\deg(Q_1) \leq \deg({\mathcal Q})$. Denote the coefficient of $Q_0=1$ in ${\mathcal Q}$ by $q_0$ (the constant coefficient).
 
 As in the proof of Theorem \ref{thm:general}, Proposition \ref{prop:I-HilbertUFD} provides nonzero $\lambda_0^\ast, \lambda_1^\ast$ in $R$ satisfying the following: for $M=\lambda_0^\ast + \lambda_1^\ast Q_1$, both $M$ and $M+{\mathcal Q}$ are irreducible in $K[\underline x]$, $\lambda_0^\ast \equiv 1-q_0\hskip 1mm [{\rm mod}\hskip 1mm  q_\infty]$ and $\lambda_1^\ast \equiv 1\hskip 1mm [{\rm mod}\hskip 1mm  \lambda_0^\ast]$ 
 (the elements $q_\infty, \lambda_0^\ast, \lambda_1^\ast$ play the respective roles of $\lambda_0^\ast, \lambda_1^\ast, \lambda_2^\ast$ from Proposition \ref{prop:I-HilbertUFD}}).

To conclude, it suffices to show that $M$ and $M+{\mathcal Q}$ are primitive. As $\lambda_0^\ast$ and $\lambda_1^\ast$ are relatively prime, $M$ is primitive. As for
$M+{\mathcal Q}$, it follows from this: the coefficients of $Q_\infty$ and $Q_0$ in $M+{\mathcal Q}$ are relatively prime. Indeed the former is  $q_\infty$ and the latter is $\lambda_0^\ast + q_0$ which is congruent to $1$ \hbox{modulo $q_\infty$.}

Finally, in the case $n=1=\deg({\mathcal Q})$, write ${\mathcal Q}=q_1x+ q_0$. We can take:
\vskip 2mm

\centerline{$\displaystyle \left\{\begin{matrix}
&\hbox{\rm if}\  q_1\not=1 \hfill &{\mathcal Q} =  [x + (q_0-1)] + [(q_1-1)x + 1] \hfill \cr
&\hbox{\rm if}\  q_1\not=-1 \hfill &{\mathcal Q} =  [-x + (q_0-1)] + [(q_1+1)x + 1] \hfill \cr
&\hbox{\rm if}\  q_1=1=-1 \hfill &{\mathcal Q} =  [rx + (rq_0+1)] + [(r+1)x + (rq_0+q_0+1)] \hfill \cr
&\hfill & \hskip 2mm \hbox{with} \ r\in R\setminus\{0,1\}. \hfill \cr
\end{matrix}\right.$}
\vskip 2mm

The more specific conclusion, alluded to in \S \ref{ssec:goldbach}, that one can further take $\deg(Q_1)=1$ if $R=K$ is a Hilbertian field, or if $R=K$ is an infinite field and $n\geq 2$, can be obtained from similar arguments but using the  Addendum to Theorem \ref{thm:schinzel1} (in \S \ref{sec:preliminary}) and Theorem \ref{thm:Main2} instead 
of Theorem \ref{thm:general}.

\subsection{Proof of Theorem \ref{thm:Main}} \label{ssec:proofMain}
Retain the notation from \S \ref{ssec:schinzel_proof}
but consider the degree $1$ case. That is, we have, 
for $i=1,\ldots,s$:

\vskip 0,5mm
\centerline{
$\left\{\begin{matrix}
P_i = A_i(\underline x) + B_i(\underline x)\hskip 1pt  y \hfill \\
F_i(\lambdabar,\xbar) = A_i(\xbar) + B_i(\xbar) \left(\sum_{j=0}^\ell \lambda_j Q_j(\underline x)\right). \\
\end{matrix}\right.$
}
\vskip 0,5mm

\noindent
Assume further that 
the polynomials $Q_i$ are the 
monomials
$Q_0,Q_1,\ldots,Q_{N_{\underline d}}$ in ${\mathcal Pol}_{k,n,\underline d}$ for some $\underline d\in (\Nn^\ast)^n$, 
with as before $Q_0=1$ and $Q_1$ and $Q_2$ monomials of  degree 
$d_1+\cdots+d_n$ and $d_1+\cdots+d_n-1$.

\begin{lemma} \label{lemma:degree1_separable} If as above $\deg_y(P_1)=\ldots=\deg_y(P_s) = 1$, then the Hilbert subset $H_K(F_1,\ldots,F_s)\subset K^{N_{\underline d}+1}$ contains a separable Hilbert subset.
\end{lemma}

\begin{proof}[Proof of Lemma \ref{lemma:degree1_separable}]
Fix $D>\max_{ { {1\leq j\leq n}\atop
{1\leq i\leq s} } }\deg_{x_j}(F_i)$ and consider the Kronecker substitution:

\vskip 1mm

\centerline{$S_D: {\mathcal Pol}_{K(\underline \lambda), n,\underline D} \rightarrow {\mathcal Pol}_{K(\underline \lambda), 1,D^{n}}$, \hskip 2mm with $\underline D=(D,\ldots,D)$,}

\vskip 1mm

\noindent
mapping  $x_j$ to $x^{D^{j-1}}$, $j=1,\ldots,n$ (introduced in \S \ref{ssec:Hilbertian-UFD}). Fix $i\in \{1,\ldots, s\}$. 
From Proposition \ref{prop:kronecker},  there exist a finite set ${\mathcal S}(F_i)$ of irreducible polynomials in $K[\underline \lambda][x]$ of degree $\geq 1$ in $x$ and a nonzero polynomial $\varphi_i \in K[\underline \lambda]$ such that the Hilbert subset ${H}_{K}(F_i)\subset K^{N_{\underline d}+1}$ contains the Hilbert subset 
$H_{K}({\mathcal S}(F_i);\varphi_i)$.
Furthermore, one can take for ${\mathcal S}(F_i)$ the set of irreducible divisors in $K[\underline \lambda][x]$ of
the following polynomial (in which $M_{\underline d} = \sum_{h=0}^{N_{{\underline d}}} \lambda_h \hskip 2pt Q_h$):

\centerline{$\displaystyle S_D(A_i+B_iM_{\underline d}) = S_D(A_i) + S_D(B_i) \sum_{h=0}^{N_{{\underline d}}} \lambda_h \hskip 2pt S_D(Q_h)$.}
The polynomials $S_D(Q_h)$ are distinct monomials in $x$ (up to multiplicative constants in $K^\times$): this indeed follows from the fact that two different integers between $0$ and $D^{n-1}-1$ have different $D$-adic expansions $a_1+a_2D+\cdots+a_{n-1}D^{n-2}$ with $0\leq a_j \leq D-1$, $j=1,\ldots,n-1$.

Note that $S_D(A_i)$ and $S_D(B_i)$ may not be relatively prime (take for example $A_i=x_2-1$ and $B_i=x_3-1$) and so Lemma \ref{lemma:first} cannot be used directly. Denote the gcd of $S_D(A_i)$ and $S_D(B_i)$ by $\Delta\in K[x]$.
Conclude from Lemma \ref{lemma:first} that the polynomial

\centerline{$\displaystyle f_i := \frac{S_D(A_i+B_iM_{\underline d})}{\Delta} = \frac{S_D(A_i)}{\Delta} + \frac{S_D(B_i)}{\Delta} \sum_{h=0}^{N_{\underline d}} \lambda_h \hskip 2pt S_D(Q_h)$}

\noindent
is irreducible in $\overline K[\underline \lambda,x]$. Since $\Delta\in K[x]$, its irreducible factors $f$ in $K[\underline \lambda,x]$ are in fact in $K[x]$, and so satisfy $H_K(f)=K^{N_{\underline d}+1}$. Conclude that one can take ${\mathcal S}(F_i)=\{f_i\}$ where $f_i$ is the polynomial displayed above.

The polynomial $f_i$ has an additional property: it is separable in $x$. Indeed, if $p>0$, not all exponents of $x$ in $f_i$ are divisible by $p$ (note that $\sum_{h=0}^{N_{{\underline d}}} \lambda_h \hskip 2pt S_D(Q_h)$ is the generic polynomial in one variable of degree $D^n-1$).

We have thus proved that the Hilbert subset $H_K(F_1,\ldots,F_s)\subset K^{N_{\underline d}+1}$ contains the separable Hilbert subset $H_K(f_1,\ldots,f_s; \varphi_1 \cdots \varphi_s)$.
\end{proof}

\begin{proof}[Proof of Theorem \ref{thm:Main}]
The statement is about polynomials in at least $2$ variables that are denoted $x_1,\ldots,x_n$  there. For consistency with the previous notation, we relabel them here $u,x_1,\ldots, x_n$, with $n\geq 1$. Set $R=k[u]$ and view $k[u,x_1,\ldots, x_n]$ as $R[\underline x]$. 

Up to adding it to the given list $(A_1,B_1),\ldots, (A_s,B_s)$ of couples of relatively prime polynomials in $R[\underline x]$, one may assume that the couple $(1,0)$ is in this list; 
this will guarantee that the desired polynomial $M$ is itself irreducible in $R[\underline x]$ as requested.

With the notation from this subsection, Lemma \ref{lemma:degree1_separable} gives that the Hilbert subset $H_K(F_1,\ldots,F_s)\subset K^{N_{\underline d}+1}$ contains a separable Hilbert subset, say $H_K(f_1,\ldots,f_s; \varphi)$.
From the separable case of Theorem \ref{lemma:hilbert2}, there is an integer $d_0$ such that for every integer $\delta \geq d_0$,
$H_K(f_1,\ldots,f_s; \varphi)$ contains a tuple $\underline \lambda^\ast 
\in R^{N_{\underline d}+1}$ such that 
$\lambda_1^\ast$ and $\lambda_2^\ast$ are relatively prime in $R$ and $\deg_u(M_{\underline d}(\underline \lambda^\ast,\underline x))=\delta$. We have {\it a fortiori} $\underline \lambda^\ast \in H_K(F_1,\ldots,F_s)\subset K^{N_{\underline d}+1}$: 
\vskip 1mm

\centerline{$F_i(\underline \lambda^\ast, \underline x) = A_i(\underline x) + B_i(\underline x) M_{\underline d}(\underline \lambda^\ast, \underline x)$ is irreducible in $K[\underline x]$, $i=1,\ldots,s$.}
\vskip 1mm

\noindent
Assume $d_0$ large enough so that,  if $d_i\geq d_0$, $i=1,\ldots,n$, then
\vskip 0,5mm

\centerline{$\displaystyle d_1+\cdots+d_n - 1 > \max_{i=1,\ldots,s} \max(\deg(A_i), \deg(B_i))$.} 
\vskip 0,5mm

\noindent
Irreduciblility of each $A_i(\underline x) + B_i(\underline x) M_{\underline d}(\underline \lambda^\ast, \underline x)$ in $R[\underline x]$ is deduced  by proving it is primitive from $\lambda_1^\ast$, $\lambda_2^\ast$ being relatively prime as in the proof of \hbox{Theorem \ref{thm:general}.}

Finally, up to multiplying $\varphi$ by the coordinate $\lambda_h$ corresponding to the monomial $x_1^{d_1} \cdots x_n^{d_n}$, one guarantees $\deg_{x_i}(M_{\underline d}(\underline \lambda^\ast, \underline x)) = d_i$, $i=1,\ldots,n$. This completes the proof: $M_{\underline d}(\underline \lambda^\ast, \underline x)$ is the requested polynomial.
\end{proof}

\begin{remark} \label{rem:degree1_Hilbertian}
Lemma \ref{lemma:degree1_separable} also shows that the degree $1$ case of the Schinzel hy\-po\-the\-sis holds when $R$ is a Hilbertian field (totally Hilbertian is not needed), thus completing the proof of the addendum to Theorem \ref{thm:schinzel1} in situation (b).
\end{remark}

\subsection{Proof of Corollary \ref{cor:spectre}} \label{sec:spectra}

Assume $n\geq 2$, fix an arbitrary field $k$, a subset  ${\mathcal S}=\{a_1,\ldots,a_t\}\subset k$, $a_0\in \overline k\setminus {\mathcal S}$, separable over $k$, and $V\in k[\underline x]$, $V\not=0$. We will show this more precise version of Corollary \ref{cor:spectre}. 
\vskip 2mm

\noindent
{\bf Corollary \ref{cor:spectre} {\rm (explicit form)}.} 
{\it Let $w_0,\ldots,w_t\in k[\underline x]$ be $t+1$ nonzero polynomials with $w_0=1$. Assume that $(w_i) +(w_j) = k[\underline x]$
for $i\not=j$ and each $w_i$ is relatively prime to $V$. For all suitably large integers $d_1,\ldots,d_n$ {\rm (}larger than some $d_0$ depending on ${\mathcal S}, a_0, V, w_1,\ldots,w_t${\rm )}, there is a polynomial $U\in k[\underline x]$ such that  these three conclusions hold:
\vskip 1mm

\noindent
\hskip -3mm
$\begin{matrix}
& \hbox{{\rm (a)} $U- a_i V = w_iH_i$ with $H_i\in k[\underline x]$ irreducible in $k(a_0)[\underline x]$} \hfill\\
& \hbox{\hskip 35mm and not dividing $w_i$, $i=1,\ldots,t$, }  \hfill\\
& \hbox{{\rm (b)} $\deg(U-a_0V) =\max(\deg(U),\deg(V))$,}  \hfill\\
& \hbox{{\rm (c)} $\deg_{x_i}(U)=d_i$, $i=1,\ldots,n$.} \hfill \\
\end{matrix}$

}
\vskip 2mm
 
In order to obtain the version of Corollary \ref{cor:spectre} from \S \ref{sec:intro}, it suffices to choose $w_1,\ldots,w_t$ as in the statement 
but not in $k$ and 
\vskip 1mm

\centerline{$d_1 > \max(\deg(V), \deg(w_1), \ldots, \deg(w_t))$.}
\vskip 1mm

\noindent
It then follows from $\deg(U)\geq \deg_{x_1}(U) = d_1$ (using (c)) that $\deg(U- a_i V) = \deg(U)$,
and next from (a) that $U- a_i V$ is reducible, $i=1,\ldots,t$.

\begin{remark} The assumption $(w_i) +(w_j) = k[\underline x]$ is necessary when $V=1$: if we have $U - a_i V = w_iH_i$ and $U - a_j V = w_jH_j$ for two distinct indices $i,j$, then $w_iH_i - w_j H_j = (a_j-a_i) V$. 
\end{remark}

\begin{proof} 
As $(w_i) +(w_j) = k[\underline x]$, $i\not=j$, the Chinese Remainder Theorem may be used to conclude that there is a polynomial $U_0\in k[\underline x]$
such that 
\vskip 1mm

\centerline{$U_0- a_iV = w_i p_i$ with $p_i\in k[\underline x]$, $i=1,\ldots,t$.}
\vskip 1mm

\noindent 
As $w_0=1$, we also have $U_0- a_0V = w_0 p_0$ 
for some $p_0$, but here $p_0$ is in $k(a_0)[\underline x]$. Furthermore the polynomials $U \in k(a_0)[\underline x]$
satisfying the same $(t+1)$ conditions are of the form
\vskip 1mm

\centerline{$\displaystyle U(\underline x) = U_0(\underline x) + M(\underline x) \prod_{i=0}^t w_i(\underline x)$}
\vskip 1mm

\noindent
for some $M\in k(a_0)[\underline x]$. For such a polynomial $U$, we have 
\vskip 1mm

\centerline{$U- a_iV= w_i \hskip 2pt (p_i + M \hskip 1pt \prod_{j\not=i} w_i(\underline x)), \hskip 2mm i=0,\ldots,t$.}
\vskip 1mm

\noindent
Up to changing $U_0$, we may assume that $p_0,\ldots,p_t$ are nonzero.

For each $i=0,\ldots,t$, the polynomials $A_i= p_i$ and $B_i = \prod_{j\not=i} w_i(\underline x)$ are relatively prime in $k(a_0)[\underline x]$. 
Namely if $\pi \in k(a_0)[\underline x]$ is a common irreducible divisor in $k(a_0)[\underline x]$ of these two polynomials, then $\pi$ divides $p_i$ and $\pi$
divides $w_j$ for some $j\not=i$ and hence, $\pi$ is a common divisor of $U_0-a_iV$ and 
$U_0-a_jV$. Therefore $\pi$ divides $V$ and $w_j$, which contradicts the assumption $(V,w_j)=1$.

Set  $R=k(a_0)[x_n]$, $K=k(a_0)(x_n)$, $\underline x=(x_1,\ldots,x_{n-1})$ and, for $\underline d\in (\Nn^\ast)^{n-1}$ and $i=0,\ldots, t$,

\vskip 1mm
\centerline{
$\left\{\begin{matrix}
P_i = A_{i}(\underline x) + B_{i}(\underline x)\hskip 1pt  y \hfill \\
F_i(\lambdabar,\xbar) = A_{i}(\xbar) + B_{i}(\xbar) \left(\sum_{j=0}^{N_{\underline d}} \lambda_j Q_j(\underline x)\right).\\
\end{matrix}\right.$
}
\vskip 1mm

\noindent
As in the proof of Theorem \ref{thm:Main}, 
the Hilbert subset $H_K(F_0,\ldots,F_t)$ contains a separable Hilbert subset 
$H_K(f_0,\ldots,f_t, \varphi)$ with $f_0,\ldots,f_t\in K[\underline \lambda, x]$ of degree $\geq 1$ in $x$ and $\varphi \in K[\underline \lambda]$, $\varphi \not=0$.

The field extension $k(a_0)/k$ is finite and separable. Setting $R_0=k[x_n]$ and $K_0 = k(x_n)$, so is the extension $K/K_0$. 
From \cite[Corollary 12.2.3]{FrJa},
 $H_K(f_0,\ldots,f_t, \varphi)$ contains a separable Hilbert subset ${\mathcal H}_{K_0}$ of $K_0^{N_{\underline d}+1}$.

Proceed as in the proof of Theorem \ref{thm:Main} to conclude that there is an integer $d_0$ with the following property:
if $\delta_1$, $\delta_2,\ldots,\delta_n$ are integers $\geq d_0$, the Hilbert subset ${\mathcal H}_{K_0}$, and so the Hilbert subset $H_K(F_0,\ldots,F_t)$ too, contains a tuple $\underline \lambda^\ast = (\lambda_0^\ast,\ldots,\lambda_{N_{\underline d}}^\ast) \in R_0^{N_{\underline d}+1}$ such that $\lambda_1^\ast$ and $\lambda_2^\ast$ are irreducible in $R_0$, and $\deg_{x_i}(M_{\underline d}(\underline \lambda^\ast,\underline x)) = \delta_{i}$, $i=1,\ldots,n$. Choosing again for $Q_1,Q_2$ monomials of respective degrees $d_1+\cdots + d_{n-1}$ and  $d_1+\cdots + d_{n-1}-1$ and assuming $d_0$ suitably large, we obtain as for Theorem \ref{thm:Main} that each of the polynomials

\vskip 1mm

\centerline{$F_i (\underline x)= A_{i}(\underline x)+ B_{i}(\underline x) M_{\underline d}(\underline \lambda^\ast, \underline x)$}
\vskip 1mm
 
\noindent
is irreducible in $k(a_0)[x_n][x_1,\ldots,x_{n-1}]$, $i=0,\ldots,t$.

Up to increasing $d_0$,  one can further guarantee that $\delta_1,\ldots,\delta_n$ are large enough so that 
$\deg(M_{\underline d}(\underline \lambda^\ast,\underline x)) > \deg(U_0)$ and $F_i$ does not divide $w_i$, $i=1,\ldots,s$.  
The polynomial

\vskip 1mm

\centerline{$U(\underline x)=U_0(\underline x)+ M_{\underline d}(\underline \lambda^\ast, \underline x) \prod_{i=0}^t w_i(\underline x) $}
\vskip 1mm
 
\noindent
is in $k[\underline x]$ and satisfies  the required condition $U- a_i V = w_iH_i$, with $H_i= F_i$ irreducible in $k(a_0)[\underline x]$, $i=0,\ldots,t$. Up to replacing the Hilbert subset $H_K(f_0,\ldots,f_t, \varphi)$ by a Zariski open subset of it, one can also request that $\deg(U-a_0V) =\max(\deg(U),\deg(V))$.
Finally $\deg_{x_i}(U)=\delta_i+ \sum_{j=1}^t \deg_{x_i}(w_j)$ can be taken to be any given suitably large integer $d_i$, $i=1,\ldots,n$. 
\end{proof}

\bibliography{progress}
\bibliographystyle{alpha}

\end{document}